\documentclass[final,leqno,onefignum,onetabnum]{siamltex1213}

\title{Analysis of a New Space-Time Parallel Multigrid Algorithm for
  Parabolic Problems}

\author{Martin J. Gander
\thanks{Section de Math\'{e}matiques
	2-4 rue du Li\`{e}vre, CP 64
	CH-1211 Gen\`{e}ve
(\email{martin.gander@unige.ch})}
\and
Martin Neum\"uller
\thanks{Inst. of Comp. Mathematics
        Altenbergerstr.~69
        4040 Linz
        Austria
(\email{martin.neumueller@jku.at})}
}

\usepackage{amssymb,latexsym,amsmath}

\newcommand{\IC}{\mathbb{C}}
\newcommand{\IN}{\mathbb{N}}
\newcommand{\IP}{\mathbb{P}}

\newcommand{\IR}{\mathbb{R}}

\newtheorem{remark}[theorem]{Remark}
\newtheorem{example}[theorem]{Example}

\newcommand{\abs}[1]{\left\lvert{#1}\right\rvert}
\newcommand{\norm}[1]{{\left\lVert{#1}\right\rVert}} 
\newcommand{\spv}[2]{{\left({#1},{#2}\right)}}

\usepackage{bbm}
\renewcommand{\i}{\mathbbm{i}}

\usepackage{bm}
\renewcommand{\vec}[1]{\bm{#1}}

\DeclareMathOperator*{\argsup}{argsup}

\usepackage{tikz}
\usepackage{pgfplots}
\usepgfplotslibrary{patchplots}

\usetikzlibrary{external}
\tikzset{external/system call={latex \tikzexternalcheckshellescape -halt-on-error
-interaction=batchmode -jobname "\image" "\texsource";
dvips -o "\image".ps "\image".dvi;
ps2eps "\image.ps"}}
\usepackage[subrefformat=parens,labelformat=parens]{subfig}

\usepackage{array}
\makeatletter
\newcommand{\thickhline}{%
    \noalign {\ifnum 0=`}\fi \hrule height 1pt
    \futurelet \reserved@a \@xhline
}
\newcolumntype{"}{@{\hskip\tabcolsep\vrule width 1pt\hskip\tabcolsep}}
\makeatother

\usepackage{multirow}
\usepackage{rotating}

\begin{document}
\maketitle
\slugger{mms}{sisc}{xx}{x}{x--x}

\begin{abstract}
  We present and analyze a new space-time parallel multigrid method
  for parabolic equations. The method is based on arbitrarily high
  order discontinuous Galerkin discretizations in time, and a finite
  element discretization in space. The key ingredient of the new
  algorithm is a block Jacobi smoother. We present a detailed
  convergence analysis when the algorithm is applied to the heat
  equation, and determine asymptotically optimal smoothing parameters,
  a precise criterion for semi-coarsening in time or full coarsening,
  and give an asymptotic two grid contraction factor estimate. We then
  explain how to implement the new multigrid algorithm in parallel, and
  show with numerical experiments its excellent strong and weak
  scalability properties.
\end{abstract}



\pagestyle{myheadings}
\thispagestyle{plain}
\markboth{}{}

\section{Introduction}
\label{s:intro}

About ten years ago, clock speeds of processors have stopped
increasing, and the only way to obtain more performance is by using
more processing cores. This has led to new generations of
supercomputers with millions of computing cores, and even today's
small devices are multicore. In order to exploit these new
architectures for high performance computing, algorithms must be
developed that can use these large numbers of cores efficiently. When
solving evolution partial differential equations, the time direction
offers itself as a further direction for parallelization, in addition
to the spatial directions, and the parareal algorithm
\cite{Lions:2001:PTD,Maday2002,Bal2005,Staff2005,Gander2007,gander:2008:nca}
has sparked renewed interest in the area of time parallelization, a
field that is now just over fifty years old, see the historical
overview \cite{Gander:2014:50Y}. We are interested here in space-time
parallel methods, which can be based on the two fundamental paradigms
of domain decomposition or multigrid. Domain decomposition methods in
space-time lead to waveform relaxation type methods, see
\cite{Gander:1998:STC,Gander:1998:WRA,Giladi:1997:STD} for classical
Schwarz waveform relaxation,
\cite{Gander:1999:OCO,Gander:2003:OSWW,Gander:2004:ABC,Gander:2007:OSW,bennequin2009homographic}
for optimal and optimized variants, and
\cite{Kwok:2014,Mandal:2014,Gander:2014:DNA} for Dirichlet-Neumann and
Neumann-Neumann waveform relaxation. The spatial decompositions can be
combined with parareal to obtain algorithms that run on arbitrary
decompositions of the space-time domain into space-time subdomains, see
\cite{Maday2005,gander:2012:psw}. Space-time multigrid methods were
developed in \cite{Hackbusch1984, Lubich1987, Vandewalle1992,
  Horton1992, Vandewalle1994, Horton1995, Horton1995a, Weinzierl2012},
and reached good F-cycle convergence behavior when appropriate
semi-coarsening and extension operators are used. For a variant for
non-linear problems, see
\cite{emmett2012toward,speck2012massively,speck2013multi}.

We present and analyze here a new space-time parallel multigrid
algorithm that has excellent strong and weak scalability properties on
large scale parallel computers. As a model problem we consider the
heat equation in a bounded domain $\Omega \subset \IR^d$, $d = 1,2,3$
with boundary $\Gamma := \partial \Omega$ on the bounded time interval
$[0,T]$,
\begin{equation}\label{chap4_modelproblem}
\begin{aligned}
 	\partial_t \, u(\vec{x},t) - \Delta u(\vec{x},t)  &=  f(\vec{x},t)
        &\quad &\text{for }
        (\vec{x},t) \in Q := \Omega \times (0,T), \\
	u(\vec{x},t) &= 0 &\quad &\text{for } (\vec{x},t) \in \Sigma:=
        \Gamma \times (0,T), \\
	u(\vec{x},0) &= u_0(\vec{x}) &\quad &\text{for } (\vec{x},t) \in
        \Sigma_0 :=\Omega \times \{0\}.
\end{aligned}
\end{equation}
We divide the time interval $[0,T]$ into subintervals
\[
  0 = t_0 < t_1 < \ldots < t_{N-1} < t_N = T, \quad \text{with } t_n = n \,
   \tau \text{ and } \tau = \frac{T}{N},
\] 
and use a standard finite element discretization in space and a
discontinuous Galerkin approximation in time, which leads to the 
large linear system in space-time
\begin{align}\label{equation1}
  \left[ K_\tau \otimes M_h + M_\tau \otimes K_h \right] \vec{u}_{n+1} =
  \vec{f}_{n+1} + N_{\tau} \otimes M_h \vec{u}_n, \quad n=0,1,\ldots,N-1.
\end{align}
Here, $M_h$ is the standard mass matrix and $K_h$ is the standard
stiffness matrix in space obtained by using the nodal basis functions
$\{\varphi_i\}_{i=1}^{N_x} \subset H_0^1(\Omega)$, i.e.
\begin{align*}
 M_h[i,j] &:= \int_\Omega \varphi_j(\vec{x}) \varphi_i(\vec{x}) \mathrm{d} \vec{x}, \quad K_h[i,j]
 := \int_\Omega \nabla \varphi_j(\vec{x}) \cdot \nabla \varphi_i(\vec{x}) \mathrm{d} \vec{x},\quad i,j = 1,\ldots,N_x. 
\end{align*}
The matrices for the time discretization, where a discontinuous
Galerkin approximation with polynomials of order $p_t \in \IN_0$ is
used, are given by
\begin{align*}
  K_\tau[k,\ell] &:= -\int_{t_{n-1}}^{t_n} \psi_\ell^n(t) \partial_t \psi_k^n(t) \mathrm{d}t
  + \psi_\ell^n(t_n) \psi_k^n(t_n),\quad k,\ell = 1,\ldots,N_t, 
\\
  M_\tau[k,\ell] &:= \int_{t_{n-1}}^{t_n} \psi_\ell^n(t) \psi_k^n(t) \mathrm{d}t, \quad
  N_\tau[k,\ell] := \psi_\ell^{n-1}(t_{n-1}) \psi_k^n(t_{n-1}).
\end{align*}
Here the basis functions for one time interval $(t_{n-1},t_n)$ are
given by $\IP^{p_t}(t_{n-1},t_n) = \mathrm{span}\{ \psi_\ell^n
\}_{\ell=1}^{N_t},$ $N_t = p_t+1$, and for $p_{t}=0$, we would for
  example get a Backward Euler scheme.  The right hand side is given
by
\begin{align*}
  \vec{f}_{n+1}[\ell N_x + j] := \int_{t_{n-1}}^{t_{n}} \int_{\Omega}
  f(\vec{x},t) \varphi_j(\vec{x}) \psi_\ell(t) \mathrm{d} \vec{x} \mathrm{d}t,
\quad j = 1,\ldots,N_x,\ \ell = 1,\ldots,N_t. 
\end{align*}
On the time interval $(t_n,t_{n+1})$, we can therefore define the
approximation
\begin{align*}
  u_h^{n+1}(\vec{x},t) = \sum_{\ell=1}^{N_t} \sum_{j=1}^{N_x} u_{\ell,j}^{n+1}\, \varphi_j(\vec{x})\,
  \psi_\ell(t), \quad \text{with } u_{\ell,j}^{n+1} := \vec{u}_{n+1}[\ell N_x + j],
\end{align*}
where $\vec{u}_{n+1}$ is the solution of the linear system
(\ref{equation1}).  We thus have to solve the  block triangular system
\begin{equation}\label{chap4_generalLinearSystem}
  \begin{pmatrix}
    A_{\tau,h} & & &  \\
    B_{\tau,h} & A_{\tau,h} & &  \\
     & \ddots & \ddots & \\
     & & B_{\tau,h} & A_{\tau,h}
  \end{pmatrix}
  \begin{pmatrix}
    \vec{u}_1 \\
    \vec{u}_2 \\
    \vdots \\
    \vec{u}_N \\
  \end{pmatrix} = 
  \begin{pmatrix}
    \vec{f}_1 \\
    \vec{f}_2 \\
    \vdots \\
    \vec{f}_N \\
  \end{pmatrix},
\end{equation}
with $A_{\tau,h} := K_\tau \otimes M_h + M_\tau \otimes K_h$ and
$B_{\tau,h} := - N_{\tau} \otimes M_h$. 

To solve the linear system (\ref{chap4_generalLinearSystem}), one can
simply apply a forward substitution with respect to the blocks
corresponding to the time steps. Hence one has to invert the matrix
$A_{\tau,h}$ for each time step, where for example a multigrid solver
can be applied. This is the usual way how time dependent problems are
solved when implicit schemes are used \cite{Thomee2006, Hairer1993,
  Hairer2010}, but this process is entirely sequential. We want to
apply a parallelizable space-time multigrid scheme to solve the global
linear system (\ref{chap4_generalLinearSystem}) at once. We present
our method in Section \ref{chap4_multigrid}, and study its properties
in Section \ref{FourierSec} using local Fourier mode
analysis. Numerical examples are given in Section \ref{NumExSec} and
the parallel implementation is discussed in Section \ref{ParallelSec},
where we also show scalability studies. We give an outlook on further
developments in Section \ref{ConclusionSec}.

\section{Multigrid method}\label{chap4_multigrid}

We present now our new space-time multigrid method to solve the linear
space-time system (\ref{chap4_generalLinearSystem}), which we rewrite
in compact form as
\begin{align}\label{chap4_linearSystem}
  \mathcal{L}_{\tau,h} \vec{u} = \vec{f}.
\end{align}
For an introduction to multigrid methods, see \cite{Hackbusch1985,
  Trottenberg2001, Vassilevski2008, Wesseling2004}.  We need a
hierarchical sequence of space-time meshes $\mathcal{T}_{N_L}$ for
$L=0,\ldots,M_L$, which has to be chosen in an appropriate way, see
Section \ref{chap4_twoGridAnalysis}.  For each space-time mesh
$\mathcal{T}_{N_L}$ we compute the system matrix
$\mathcal{L}_{\tau_L,h_L}$ for $L=0,\ldots,M_L$. On the last (finest)
level $M_L$, we have to solve the original system
(\ref{chap4_linearSystem}), i.e.  $\mathcal{L}_{\tau_{M_L},h_{M_L}} =
\mathcal{L}_{\tau,h}$.

We denote by $\mathcal{S}_{\tau_L,h_L}^\nu$ the damped block Jacobi
smoother with $\nu \in \IN$ steps,
\begin{align}\label{chap4_smoother}
\quad\vec{u}^{k+1} = \vec{u}^k + \omega_t (\widetilde{D}_{\tau_L,h_L})^{-1} \left[
  \vec{f} -  \mathcal{L}_{\tau_L,h_L} \vec{u}^{k} \right].
\end{align}
Here $\widetilde{D}_{\tau_L,h_L}^{-1}$ denotes an approximation of the
inverse of the block diagonal matrix $D_{\tau_L,h_L} :=
\mathrm{diag}\{ A_{\tau_L,h_L} \}_{n=1}^{N_L}$, where a block
$A_{\tau_L,h_L}:=M_{h_L} \otimes K_{\tau_L} + K_{h_L} \otimes
M_{\tau_L}$ corresponds to one time step. We will consider in
particular approximating $(D_{\tau_L,h_L})^{-1}$ by applying one
multigrid V-cycle in space at each time step, using a
standard tensor product multigrid, like in \cite{Boerm2001}.  

For the prolongation operator $\mathcal{P}^L$ we use the standard
interpolation from coarse space-time grids to the next finer
space-time grids. The prolongation operator will thus depend on the
space-time hierarchy chosen. The restriction operator is the adjoint
of the prolongation operator, $\mathcal{R}^L =
(\mathcal{P}^L)^\top$. With $\nu_1, \nu_2 \in \IN$ we denote the
number of pre- and post smoothing steps, and $\gamma \in \IN$ defines
the cycle index, where typical choices are $\gamma = 1$ (V-cycle), and
$\gamma = 2$ (W-cycle). On the coarsest level $L=0$ we solve the linear
system, which consists of only one time step, exactly by using an
LU-factorization for the system matrix $\mathcal{L}_{\tau_0,h_0}$. For
a given initial guess we apply this space-time multigrid cycle several
times, until we have reached a given relative error reduction
$\varepsilon_{\mathrm{MG}}$.

To study the convergence behavior of our space-time multigrid method,
we use local Fourier mode analysis. This type of analysis was used in
\cite{Gander:2014:AOA} to study a two-grid cycle for an ODE model
problem, and we will need the following definitions and results, whose
proof can be found in \cite{Gander:2014:AOA}.
\begin{theorem}[Discrete Fourier transform]\label{chap4_theorem1}
  For $m \in \IN$ let $\vec{u} \in \IR^{2 m}$. Then
  \begin{align*}
    \vec{u} = \sum_{k = 1-m}^{m} \hat{u}_k \vec{\varphi}(\theta_k), \qquad
    \vec{\varphi}(\theta_k)[\ell]:= e^{\i \ell \theta_k},\quad \ell = 1,\ldots,
    2m,  \qquad \theta_k := \frac{k \pi}{m},
  \end{align*}
  with the coefficients
  \begin{align*}
    \hat{u}_k := \frac{1}{2m} \spv{\vec{u}}{\vec{\varphi}(-\theta_k)}_{\ell^2} =
    \frac{1}{2m} \sum_{\ell = 1}^{2m} \vec{u}[\ell]
    \vec{\varphi}(-\theta_k)[\ell], \qquad \text{for } k = 1-m,\ldots,m.
  \end{align*}
\end{theorem}

\begin{definition}[Fourier modes, Fourier frequencies]
  Let $N_L \in \IN$. Then the vector valued function
  $\vec{\varphi}(\theta_k)[\ell]:= e^{\i \ell \theta_k}$, $\ell = 1,\ldots,N_L$
  is called Fourier mode with frequency
  \begin{align*}
  \theta_k \in \Theta_L := \left\{ \frac{2 k \pi}{N_L} : k = 1 -
    \frac{N_L}{2},\ldots,\frac{N_L}{2} \right\} \subset (-\pi,\pi].
  \end{align*}
  The frequencies $\Theta_L$ are further separated into low and high
  frequencies
  \begin{align*}
     \Theta_L^{\mathrm{low}} &:= \Theta_L \cap (-\frac{\pi}{2},\frac{\pi}{2}],\\
     \Theta_L^{\mathrm{high}} &:= \Theta_L \cap \left( (-\pi,-\frac{\pi}{2}]
       \cup (\frac{\pi}{2},\pi] \right) = \Theta_L \setminus
     \Theta_L^{\mathrm{low}}.
  \end{align*}
\end{definition}

\begin{definition}[Fourier space]
  For $N_{L}, N_t \in \IN$ let the vector $\vec{\Phi}^L(\theta_k) \in \IC^{N_t
    N_L}$ be defined as in Lemma \ref{chap4_lemma1} with frequency $\theta_k
  \in \Theta_L$. Then we define the linear space of Fourier modes with
  frequency $\theta_k$ as
  \begin{align*}
    \Psi_L(\theta_k) &:= \mathrm{span}\left\{ \vec{\Phi}^L(\theta_k) \right\} \\
    &\phantom{:}= \left\{ \vec{\psi}^L(\theta_k) \in \IC^{N_t N_L} :
      \vec{\psi}_n^L(\theta_k) =
      U \vec{\Phi}_n^L(\theta_k),\; n = 1,\ldots,N_L \text{ and } U \in \IC^{N_t
        \times N_t} \right\}.
\end{align*}
\end{definition}

\begin{definition}[Space of harmonics]
   For $N_{L}, N_t \in \IN$ and for a low frequency $\theta_k\in
   \Theta_L^{\mathrm{low}}$ let the vector $\vec{\Phi}^L(\theta_k) \in \IC^{N_t
     N_L}$ be defined as in Lemma \ref{chap4_lemma1}. Then the linear space of
   harmonics with frequency $\theta_k$ is given by
  \begin{align*}
    \mathcal{E}_L(\theta_k) &:= \mathrm{span}\left\{ \vec{\Phi}^L(\theta_k),
      \vec{\Phi}^L(\gamma(\theta_k)) \right\} \\
    &\phantom{:}= \big\{ \vec{\psi}^L(\theta_k) \in \IC^{N_t N_L} :
    \vec{\psi}_n^L(\theta_k) =
      U_1 \vec{\Phi}_n^L(\theta_k) + U_2 \vec{\Phi}_n^L(\gamma(\theta_k)), \\
      &\qquad\qquad\qquad\qquad\qquad\qquad n = 1,\ldots,N_L \text{ and } U_1,
      U_2 \in
      \IC^{N_t \times N_t} \big\}.
\end{align*}
\end{definition}

\begin{lemma}\label{chap4_lemma1}
  The vector $\vec{u} = (\vec{u}_1, \vec{u}_2,\ldots,\vec{u}_{N_L})^\top \in
  \IR^{N_L\,N_t}$ for $N_{L-1}, N_t \in \IN$ and $N_L = 2 N_{L-1}$ can be written
  as
  \[ \vec{u} = \sum_{k=-N_{L-1}+1}^{N_{L-1}} \vec{\psi}^L(\theta_k, U) =
  \sum_{\theta_k \in \Theta_L} \vec{\psi}^L(\theta_k,U), \]
  with the vectors
  $\vec{\psi}_n^L(\theta_k, U) := U \vec{\Phi}_n^L(\theta_k)$ and
  $\vec{\Phi}_n^L(\theta_k)[\ell] := \vec{\varphi}(\theta_k)[n]$
  for $n=1,\ldots,N_L$ and $\ell = 1,\ldots,N_t$,
  and the coefficient matrix
  $U = \mathrm{diag}(\hat{u}_k[1],\ldots,\hat{u}_k[N_t]) \in \IC^{N_t \times N_t}$
  with the coefficients
  $\hat{u}_k[\ell] := \frac{1}{N_{L}} \sum_{i=1}^{N_L} u_{i}[\ell]
  \vec{\varphi}(-\theta_k)[i]$ for $k = 1 -N_{L-1},\ldots,N_{L-1}$.
\end{lemma}
\begin{lemma}\label{chap4_theorem2}
  For $\lambda \in \IC$ the eigenvalues of the matrix $(K_{\tau_L} + \lambda M_{\tau_L})^{-1} N_{\tau_L} \in
  \IC^{N_t \times N_t}$ are given by
  \begin{align*}
    \sigma((K_{\tau_L} + \lambda  M_{\tau_L})^{-1} N_{\tau_L}) = \{ 0, R(\lambda  \tau_L) \},
  \end{align*}
  where $R(z)$ is the A-stability function of
  the given discontinuous Galerkin time stepping scheme. In particular the A-stability function $R(z)$ is given by the $(p_t,p_t+1)$ subdiagonal Pad\'{e}  approximation of the exponential function $e^{z}$.
\end{lemma}
\begin{lemma}\label{chap4_lemma9}
  The mapping $\gamma : \Theta_L^{\mathrm{low}} \rightarrow
  \Theta_L^{\mathrm{high}}$ with
  $\gamma(\theta_k) := \theta_k - \mathrm{sign}(\theta_k) \pi$
  is a one to one mapping.
\end{lemma}
\begin{lemma}\label{chap4_lemma12}
  Let $\theta_k \in \Theta_L^{\mathrm{low}}$. Then the restriction
  operator $\mathcal{R}^L$ as defined in (\ref{chap4_restrictionTime}) has the
  mapping property
  \begin{align*}
    \mathcal{R}^L : \mathcal{E}_L(\theta_k) \rightarrow \Psi_{L-1}(2 \theta_k),
  \end{align*}
  with the mapping
  \begin{align*}
    \begin{pmatrix} U_1\\U_2 \end{pmatrix} \mapsto \begin{pmatrix}
    \hat{\mathcal{R}}(\theta_k) &
    \hat{\mathcal{R}}(\gamma(\theta_k)) \end{pmatrix} \begin{pmatrix}
    U_1\\U_2 \end{pmatrix} \in \IC^{N_t \times N_t}
  \end{align*}
  and the Fourier symbol 
  \begin{align*}
    \hat{\mathcal{R}}(\theta_k) := e^{-\i \theta_k} R_1 + R_2.
  \end{align*}
\end{lemma}
\begin{lemma}\label{chap4_lemma13}
  Let $\theta_k \in \Theta_L^{\mathrm{low}}$. Then the the
  prolongation operator $\mathcal{P}^L$ as defined in (\ref{chap4_restrictionTime}) has the
  mapping property 
  \begin{align*}
    \mathcal{P}^L : \Psi_{L-1}(2 \theta_k) \rightarrow \mathcal{E}_L(\theta_k),
  \end{align*}
  with the mapping
  \begin{align*}
    U \mapsto \begin{pmatrix}
    \hat{\mathcal{P}}(\theta_k) \\
    \hat{\mathcal{P}}(\gamma(\theta_k)) \end{pmatrix} U \in \IC^{2 N_t \times N_t}
  \end{align*}
  and the Fourier symbol 
  \begin{align*}
    \hat{\mathcal{P}}(\theta_k) := \frac{1}{2}\left[e^{\i \theta_k} R_1^\top +
      R_2^\top\right].
  \end{align*}
\end{lemma}
\begin{lemma}\label{chap4_lemma14}
  The frequency mapping 
  \[ \beta: \Theta_L^{\mathrm{low}} \rightarrow \Theta_{L-1}\qquad \text{with}
  \qquad \theta_k \mapsto 2 \theta_k \]
  is a one to one mapping.
\end{lemma}

\section{Local Fourier mode analysis}\label{FourierSec}

For simplicity we assume that $\Omega = (0,1)$ is a one-dimensional
domain, which is divided into uniform elements with mesh size $h$.
The analysis for higher dimensions is more technical, but the tools
stay the same as for the one dimensional case.  The standard one
dimensional mass and stiffness matrices are
\begin{align*}
  M_h = \frac{h}{6}
\begin{pmatrix}
  4 & 1 &     & \\
  1 & 4 & \ddots &   \\
  &\ddots & \ddots &1 \\
   &   & 1 & 4
 \end{pmatrix}, \qquad   K_h = \frac{1}{h}
\begin{pmatrix}
  2 & -1 &   &   \\
  -1 & 2 & \ddots &   \\
  &\ddots & \ddots & -1\\
  & & -1 & 2 
 \end{pmatrix}.
\end{align*}

\subsection{Smoothing analysis}

The iteration matrix of damped block Jacobi is 
\begin{align*}
  \mathcal{S}_{\tau_L,h_L}^\nu = \left[I - \omega_t
    (D_{\tau_L,h_L})^{-1} \mathcal{L}_{\tau_L,h_L} \right]^\nu,
\end{align*}
where $D_{\tau_L,h_L}$ is a block diagonal matrix with blocks
$A_{\tau_L,h_L}$. We first use the exact inverse of the diagonal
matrix $D_{\tau_L,h_L}$ in our analysis, the V-cycle approximation is
studied later, see Remark \ref{V-cycleApproxSpace}. We denote by
$N_{L_t} \in \IN$ the number of time steps and by $N_{L_x} \in \IN$
the degrees of freedom in space for level $L \in \IN_0$. Using
Theorem \ref{chap4_theorem1}, we can prove
\begin{lemma}\label{chap4_lemma15}
  Let $\vec{u} = (\vec{u}_1,\vec{u}_2,\ldots,\vec{u}_{N_{L_t}})^\top \in \IR^{N_t
    N_{L_x} N_{L_t}}$ for $N_t, N_{L_x}, N_{L_t} \in \IN$, where we assume that
  $N_{L_x}$ and $N_{L_t}$ are even numbers, and assume that
  \[ \vec{u}_n \in \IR^{N_t N_{L_x}} \qquad \text{and} \qquad \vec{u}_{n,r} \in
  \IR^{N_t} \]
  for $n=1,\ldots,N_{L_t}$ and $r=1,\ldots,N_{L_x}$. Then the vector $\vec{u}$ can be
  written as
  \begin{align*}
    \vec{u} = \sum_{\theta_x \in \Theta_{L_x} }\sum_{\theta_t \in \Theta_{L_t}}
    \vec{\psi}^{L_x,L_t}(\theta_x,\theta_t)
  \end{align*}
  with the vectors
  \begin{align*}
    \vec{\psi}_{n,r}^{L_x,L_t}(\theta_x,\theta_t) := U
    \vec{\Phi}_{n,r}^{L_x,L_t}(\theta_x,\theta_t),\qquad
    \vec{\Phi}_{n,r}^{L_x,L_t}(\theta_x,\theta_t) :=
    \vec{\Phi}_{n}^{L_t}(\theta_t) \vec{\varphi}^{L_x}(\theta_x)[r]
  \end{align*}
   for $n=1,\ldots,N_{L_t}$, $r=1,\ldots,N_{L_x}$ and with the coefficient matrix
  \begin{align*}
    U := \mathrm{diag}\left( \hat{u}_{x,t}[1],\ldots,\hat{u}_{x,t}[N_t] \right)
    \in \IC^{N_t \times N_t}
  \end{align*}
  with the coefficients for $\theta_x \in \Theta_{L_x}$ and $\theta_t \in
  \Theta_{L_t}$
  \begin{align*}
    \hat{u}_{x,t}[\ell] := \frac{1}{N_{L_x}} \frac{1}{N_{L_t}} \sum_{r = 1}^{N_{L_x}}
    \sum_{n=1}^{N_{L_t}} \vec{u}_{n,r}[\ell] \vec{\varphi}(-\theta_x)[r]
    \vec{\varphi}(-\theta_t)[n].
  \end{align*}
\end{lemma}
\begin{proof}
  For $\vec{u} = (\vec{u}_1,\vec{u}_2,\ldots,\vec{u}_{N_{L_t}})^\top \in \IR^{N_t
    N_{L_x} N_{L_t}}$ we define for $s = 1,\ldots,N_{L_x}$ the vector
  $\vec{w}^s \in \IR^{N_{L_t} N_t}$ as $\vec{w}_n^s[\ell] :=
  \vec{u}_{n,s}[\ell]$. Applying Lemma \ref{chap4_lemma1} to the vector
  $\vec{w}^s$ results in
  \begin{align*}
    \vec{u}_{i,s}[\ell] = \vec{w}_i^s[\ell] = \sum_{\theta_t \in \Theta_{L_t}}
    \vec{\psi}^{L_t}(\theta_t) = \sum_{\theta_t \in \Theta_{L_t}}
    U_t[\ell,\ell] \vec{\varphi(\theta_t)}[i],
  \end{align*}
  with
  \begin{align*}
    U_t[\ell,\ell] = \hat{w}_t^s[\ell] = \frac{1}{N_L} \sum_{n=1}^{N_{L_t}}
    \vec{u}_{n,s}[\ell] \vec{\varphi}(-\theta_t)[n].
  \end{align*}
  Next, we define for a fixed $n \in \left\{
  1,\ldots,N_{L_t}\right\}$ and a fixed $\ell \in \left\{ 1,\ldots,N_t
  \right\}$ the vector $\vec{z}^{n,\ell} \in \IR^{N_{L_x}}$ as
  $\vec{z}^{n,\ell}[s] := u_{n,s}[\ell]$. Applying Theorem
  \ref{chap4_theorem1} to the vector $\vec{z}^{n,\ell}$,
  we get for $s = 1,\ldots,N_{L_x}$
  \begin{align*}
    u_{n,s}[\ell] = \vec{z}^{n,\ell}[s] = \sum_{\theta_x \in \Theta_{L_x}}
    \hat{z}_x^{n,\ell} \vec{\varphi}(\theta_x)[s],
  \quad\mbox{with}\quad
    \hat{z}_x^{n,\ell} = \frac{1}{N_{L_x}} \sum_{r = 1}^{N_{L_x}}
    \vec{u}_{n,r}[\ell]\vec{\varphi}(-\theta_x)[r].
  \end{align*}
  Combining the results above, we obtain the statement of this lemma with
  \begin{align*}
    \vec{u}_{i,s}[\ell] &= \sum_{\theta_x \in \Theta_{L_x}} \sum_{\theta_t \in
      \Theta_{L_t}} \vec{\varphi}(\theta_x)[s] \vec{\varphi(\theta_t)}[i]
    \frac{1}{N_{L_x}} \frac{1}{N_{L_t}} \sum_{r = 1}^{N_{L_x}}  \sum_{n=1}^{N_{L_t}}
    \vec{u}_{n,r}[\ell]\vec{\varphi}(-\theta_x)[r]  \vec{\varphi}(-\theta_t)[n]\\
    &= \sum_{\theta_x \in \Theta_{L_x}} \sum_{\theta_t \in
      \Theta_{L_t}} \hat{u}_{x,t}[\ell] \vec{\varphi}(\theta_x)[s]
    \vec{\varphi(\theta_t)}[i]\\
    &= \sum_{\theta_x \in \Theta_{L_x}} \sum_{\theta_t \in
      \Theta_{L_t}} U[\ell,\ell]
    \vec{\Phi}_{i,s}^{L_x,L_t}(\theta_x,\theta_t)[\ell]\\
    &= \sum_{\theta_x \in \Theta_{L_x}} \sum_{\theta_t \in
      \Theta_{L_t}} \vec{\psi}_{i,s}^{L_x,L_t}(\theta_x,\theta_t)[\ell].
  \end{align*}
\end{proof}
\begin{definition}[Fourier space]
  For $N_t, N_{L_x}, N_{L_t} \in \IN$ and the frequency  $\theta_x \in
  \Theta_{L_x}$ and $\theta_t \in \Theta_{L_t}$, let the vector
  $\vec{\Phi}^{L_x,L_t}(\theta_x,\theta_t) \in \IC^{N_t N_{L_x} N_{L_t}}$ be
  as in Lemma \ref{chap4_lemma15}. Then we define the linear space of
  Fourier modes with frequencies $(\theta_x, \theta_t)$ as
  \begin{align*}
    \Psi_{L_x,L_t}(\theta_x,\theta_t) &:= \mathrm{span} \left\{
      \vec{\Phi}^{L_x,L_t}(\theta_x,\theta_t) \right\}\\
    &\phantom{:}= \big\{ \vec{\psi}^{L_x,L_t}(\theta_x,\theta_t) \in \IC^{N_t
    N_{L_x} N_{L_t}} : \vec{\psi}_{n,r}^{L_x,L_t}(\theta_x,\theta_t) := U
    \vec{\Phi}_{n,r}^{L_x,L_t}(\theta_x,\theta_t), \\
    &\qquad \qquad \qquad \qquad \quad n = 1,\ldots,N_{L_t}, r =
    1,\ldots,N_{L_x} \text{ and } U \in \IC^{N_t \times N_t} \big\}.
  \end{align*}
\end{definition}
\begin{lemma}[Shifting equality]\label{chap4_lemma16}
  For $N_t, N_{L_x}, N_{L_t} \in \IN$ and the frequencies  $\theta_x \in
  \Theta_{L_x}$, $\theta_t \in \Theta_{L_t}$ let
  $\vec{\psi}^{L_x,L_t}(\theta_x,\theta_t) \in
  \Psi_{L_x,L_t}(\theta_x,\theta_t)$. Then we have the shifting equalities
  \begin{align*}
    \vec{\psi}_{{n-1},r}^{L_x,L_t}(\theta_x,\theta_t) = e^{-\i \theta_t}
    \vec{\psi}_{n,r}^{L_x,L_t}(\theta_x,\theta_t),\quad
        \vec{\psi}_{n,r-1}^{L_x,L_t}(\theta_x,\theta_t) = e^{-\i \theta_x}
    \vec{\psi}_{n,r}^{L_x,L_t}(\theta_x,\theta_t)
  \end{align*}
  for $n = 2,\ldots,N_{L_t}$ and $r = 2,\ldots,N_{L_x}$.
\end{lemma}
\begin{proof}
  The result follows from the fact that
  \[  \vec{\varphi(\theta)}[n-1]= e^{\i (n-1) \theta} = e^{- \i \theta} e^{\i n
    \theta} =  e^{- \i \theta} \vec{\varphi(\theta)}[n],\]
  which can be applied for the frequencies in space $\theta_x \in \Theta_{L_x}$
  and the frequencies in time $\theta_t \in \Theta_{L_t}.$
\end{proof}
\begin{lemma}[Fourier symbol of $\mathcal{L}_{\tau_L,h_L}$]\label{chap4_lemma17}
  For the frequencies $\theta_x \in \Theta_{L_x}$ and $\theta_t \in \Theta_{L_t}$
  we consider the vector $\vec{\psi}^{L_x,L_t}(\theta_x,\theta_t) \in
  \Psi_{L_x,L_t}(\theta_x,\theta_t)$. Then for $n = 2,\ldots,N_{L-t}$ and $r =
  2,\ldots,N_{L_x}-1$ we have
  \begin{align*}
    \left( \mathcal{L}_{\tau_L,h_L}  \vec{\psi}^{L_x,L_t}(\theta_x,\theta_t)
    \right)_{n,r} =  \hat{\mathcal{L}}_{\tau_L,h_L}(\theta_x, \theta_t)
    \vec{\psi}_{n,r}^{L_x,L_t}(\theta_x,\theta_t),
  \end{align*}
  where the Fourier symbol is given by
  \begin{align*}
    \hat{\mathcal{L}}_{\tau_L,h_L}(\theta_x, \theta_t) := \frac{h_L}{3} \left(2 +
      \cos(\theta_x)\right) \left[ K_{\tau_L} + h_L^{-2} \beta(\theta_x)
      M_{\tau_L} - e^{-\i \theta_t} N_{\tau_L}\right] \in
      \IC^{N_t\times N_t},
  \end{align*}
  with the function
  $\beta(\theta_x) := 6\; \frac{1 - \cos(\theta_x)}{2 + \cos(\theta_x)} \in
    [0,12]$.
\end{lemma}
\begin{proof}
  Let  $\vec{\psi}^{L_x,L_t}(\theta_x,\theta_t) \in
  \Psi_{L_x,L_t}(\theta_x,\theta_t)$. Then we have for $n = 2,\ldots,N_{L_t}$
  and using Lemma \ref{chap4_lemma16}
  \begin{align*}
    \left( \mathcal{L}_{\tau_L,h_L}  \vec{\psi}^{L_x,L_t}(\theta_x,\theta_t)
    \right)_{n} &=  B_{\tau_L,h_L} \vec{\psi}_{n-1}^{L_x,L_t}(\theta_x,\theta_t) +
    A_{\tau_L,h_L} \vec{\psi}_{n}^{L_x,L_t}(\theta_x,\theta_t)\\
    &= \left( e^{-\i \theta_t} B_{\tau_L,h_L} + A_{\tau_L,h_L} \right)
    \vec{\psi}_{n}^{L_x,L_t}(\theta_x,\theta_t).
  \end{align*}
  Hence, we have to study the action of $A_{\tau,h}$ and $B_{\tau,h}$ on the
  local vector $\vec{\psi}_{n}^{L_x,L_t}(\theta_x,\theta_t)$. By using the
  definition of $B_{\tau,h}$, we obtain for $r = 2,\ldots,N_{L_x}-1$ and $\ell =
  1,\ldots,N_t$ using Lemma \ref{chap4_lemma16}
  \begin{align*}
     &\left( B_{\tau_L,h_L} \vec{\psi}_n^{L_x,L_t}(\theta_x,\theta_t)
    \right)_{r}[\ell] = - \sum_{s=1}^{N_{L_x}} \sum_{k=1}^{N_t} M_{h_L}[r,s]
    N_{\tau_L}[\ell,k] \vec{\psi}_{n,s}^{L_x,L_t}(\theta_x,\theta_t)[k]\\
    &\qquad= -  \sum_{k=1}^{N_t} \frac{h_L}{6} \left(
      \vec{\psi}_{n,r-1}^{L_x,L_t}(\theta_x,\theta_t)[k] +
      4 \vec{\psi}_{n,r}^{L_x,L_t}(\theta_x,\theta_t)[k] +
      \vec{\psi}_{n,r+1}^{L_x,L_t}(\theta_x,\theta_t)[k] \right)
    N_{\tau_L}[\ell,k]\\
    &\qquad= -\frac{h_L}{6}  \sum_{k=1}^{N_t} N_{\tau_L}[\ell,k] \left( e^{-\i
        \theta_x} + 4 + e^{\i \theta_x} \right)
    \vec{\psi}_{n,r}^{L_x,L_t}(\theta_x,\theta_t)[k]\\
    &\qquad= -\frac{h_L}{3} \left(2 + \cos(\theta_x)\right) \sum_{k=1}^{N_t}
    N_{\tau_L}[\ell,k]\vec{\psi}_{n,r}^{L_x,L_t}(\theta_x,\theta_t)[k]\\
    &\qquad= -\frac{h_L}{3} \left(2 + \cos(\theta_x)\right)  \left(
      N_{\tau_L}\vec{\psi}_{n,r}^{L_x,L_t}(\theta_x,\theta_t) \right)[\ell].
  \end{align*}
  Next we study the action of the matrix $A_{\tau,h}$ on the local vector
  $\vec{\psi}_{n}^{L_x,L_t}(\theta_x,\theta_t)$:
  \begin{align*}
    &\left( A_{\tau_L,h_L} \vec{\psi}_n^{L_x,L_t}(\theta_x,\theta_t)
    \right)_{r}[\ell] =  \sum_{s=1}^{N_{L_x}} \sum_{k=1}^{N_t}  M_{h_L}[r,s]
    K_{\tau_L}[\ell,k] \vec{\psi}_{n,s}^{L_x,L_t}(\theta_x,\theta_t)[k]\\
    &\qquad\qquad\qquad\qquad\qquad\qquad+ \sum_{s=1}^{N_{L_x}} \sum_{k=1}^{N_t}
    K_{h_L}[r,s] M_{\tau_L}[\ell,k]
    \vec{\psi}_{n,s}^{L_x,L_t}(\theta_x,\theta_t)[k]\\
    &\quad= \frac{h_L}{3} \left( 2 + \cos(\theta_x) \right)\sum_{k=1}^{N_t}
    K_{\tau_L}[\ell,k] \vec{\psi}_{n,s}^{L_x,L_t}(\theta_x,\theta_t)[k]\\
    &\quad\quad+\sum_{k=1}^{N_t} \frac{1}{h_L} \left(
      -\vec{\psi}_{n,r-1}^{L_x,L_t}(\theta_x,\theta_t)[k] +
      2 \vec{\psi}_{n,r}^{L_x,L_t}(\theta_x,\theta_t)[k] -
      \vec{\psi}_{n,r+1}^{L_x,L_t}(\theta_x,\theta_t)[k] \right)
    M_{\tau_L}[\ell,k]\\
    &\quad=  \frac{h_L}{3} \left( 2 + \cos(\theta_x) \right) \left( K_{\tau_L}
      \vec{\psi}_{n,r}^{L_x,L_t}(\theta_x,\theta_t) \right)[\ell]\\
    &\quad\quad+\frac{2}{h_L} (1 - \cos(\theta_x)) \sum_{k=1}^{N_t}
    M_{\tau_L}[\ell,k] \vec{\psi}_{n,r}^{L_x,L_t}(\theta_x,\theta_t)[k]\\
    &\quad=\left(\left[ \frac{h_L}{3} \left( 2 + \cos(\theta_x) \right)
        K_{\tau_L} + \frac{2}{h_L} (1 - \cos(\theta_x))  M_{\tau_L} \right]
    \vec{\psi}_{n,r}^{L_x,L_t}(\theta_x,\theta_t) \right)[\ell],
  \end{align*}
  where we used Lemma \ref{chap4_lemma16}. Hence we conclude the proof with
  \begin{align*}
    &\left( \mathcal{L}_{\tau_L,h_L}  \vec{\psi}^{L_x,L_t}(\theta_x,\theta_t)
    \right)_{n,r} =  \frac{h_L}{3} \left( 2 + \cos(\theta_x) \right) \left(
      K_{\tau_L} - e^{-\i \theta_t} N_{\tau_L} \right)
    \vec{\psi}_{n,r}^{L_x,L_t}(\theta_x,\theta_t)\\
    &\qquad\qquad\qquad\qquad\qquad\quad+ \frac{2}{h_L} (1 - \cos(\theta_x))
    M_{\tau_L} \vec{\psi}_{n,r}^{L_x,L_t}(\theta_x,\theta_t)\\
    &\qquad\quad= \frac{h_L}{3} \left( 2 + \cos(\theta_x) \right) \left(
      K_{\tau_L} + 6 h_L^{-2} \frac{1 - \cos(\theta_x)}{2 + \cos(\theta_x)}
      M_{\tau_L} - e^{-\i \theta_t} N_{\tau_L} \right)
    \vec{\psi}_{n,r}^{L_x,L_t}(\theta_x,\theta_t)\\
    &\qquad\quad= \frac{h_L}{3} \left( 2 + \cos(\theta_x) \right) \left(
      K_{\tau_L} + h_L^{-2} \beta(\theta_x) M_{\tau_L} - e^{-\i \theta_t} N_{\tau_L}
    \right) \vec{\psi}_{n,r}^{L_x,L_t}(\theta_x,\theta_t).
  \end{align*}
\end{proof}

\noindent If we assume periodic boundary conditions in space-time, i.e.
\begin{equation}\label{chap4_periodicBCST}
\begin{aligned}
  u(t,0) &= u(t,1) &\qquad \text{for } t &\in (0,T),\\
  u(0,x) &= u(T,x) &\qquad \text{for } x &\in \Omega = (0,1),
\end{aligned}
\end{equation}
we obtain from Lemma \ref{chap4_lemma17} the mapping property
\begin{equation}\label{chap4_mappingProps1}
\begin{aligned}
  \mathcal{L}_{\tau_L,h_L}: \Psi_{L_x,L_t}(\theta_x,\theta_t) &\rightarrow
  \Psi_{L_x,L_t}(\theta_x,\theta_t),\\
  U &\mapsto \hat{\mathcal{L}}_{\tau_L,h_L}(\theta_x,
  \theta_t) U.
\end{aligned}
\end{equation}
\begin{lemma}[Mapping property of $\mathcal{S}_{\tau_L,h_L}^\nu$]
\label{chap4_lemma18}
    For the frequencies $\theta_x \in \Theta_{L_x}$ and $\theta_t \in
    \Theta_{L_t}$ we consider the vector
    $\vec{\psi}^{L_x,L_t}(\theta_x,\theta_t) \in
    \Psi_{L_x,L_t}(\theta_x,\theta_t)$. Then under the assumption of
    periodic boundary conditions (\ref{chap4_periodicBCST}), we have
    for $n = 1,\ldots,N_{L_t}$ and $r = 1,\ldots,N_{L_x}$
  \begin{align*}
    \left( \mathcal{S}_{\tau_L,h_L}^\nu  \vec{\psi}^{L_x,L_t}(\theta_x,\theta_t)
    \right)_{n,r} =  \left[\hat{\mathcal{S}}_{\tau_L,h_L}( \theta_x, \theta_t)\right]^\nu
    \vec{\psi}_{n,r}^{L_x,L_t}(\theta_x,\theta_t),
  \end{align*}
   where the Fourier symbol is given by
  \begin{align*}
    \hat{\mathcal{S}}_{\tau_L,h_L}(\theta_x, \theta_t) := (1-\omega_t) I_{N_t}
    + \omega_t e^{-\i \theta_t} \left( K_{\tau_L} + h_L^{-2} \beta(\theta_x)
      M_{\tau_L} \right)^{-1} N_{\tau_L} \in \IC^{N_t \times N_t},
  \end{align*}
  with the function $\beta(\theta_x)$ as defined in Lemma \ref{chap4_lemma17}.
\end{lemma}
\begin{proof}
   Let  $\vec{\psi}^{L_x,L_t}(\theta_x,\theta_t) \in
  \Psi_{L_x,L_t}(\theta_x,\theta_t)$, then for a fixed $n = 1,\ldots,N_{L_t}$
  and a fixed $r = 1,\ldots,N_{L_x}$ we have that
  \begin{align*}
     \left( \mathcal{S}_{\tau_L,h_L}^1  \vec{\psi}^{L_x,L_t}(\theta_x,\theta_t)
    \right)_{n,r} &= \left(\left( I_{N_t N_{L_x} N_{L_t}} - \omega_t
        (D_{\tau_L,h_L})^{-1} \mathcal{L}_{\tau_L,h_L} \right)
      \vec{\psi}^{L_x,L_t}(\theta_x,\theta_t) \right)_{n,r}\\
  &= \left( I_{N_t} - \omega_t \left( \hat{A}_{\tau_L,h_L}(\theta_x)
    \right)^{-1} \hat{\mathcal{L}}_{\tau_L,h_L}(\theta_x, \theta_t)
  \right)\vec{\psi}_{n,r}^{L_x,L_t}(\theta_x,\theta_t)\\
  &=: \hat{\mathcal{S}}_{\tau_L,h_L}( \theta_x, \theta_t)
  \vec{\psi}_{n,r}^{L_x,L_t}(\theta_x,\theta_t)
  \end{align*}
   with
  \begin{align*}
    \hat{A}_{\tau_L,h_L}(\theta_x) :&=  \frac{h_L}{3} \left( 2 + \cos(\theta_x)
    \right) K_{\tau_L} + \frac{2}{h_L} (1 - \cos(\theta_x)) M_{\tau_L}\\
    &= \frac{h_L}{3} \left( 2 + \cos(\theta_x)
    \right) \left[ K_{\tau_L} + h_L^{-2} \beta(\theta_x) M_{\tau_L} \right].
  \end{align*}
  Further calculations give
  \begin{align*}
    \left( \hat{A}_{\tau_L,h_L}(\theta_x)
    \right)^{-1} \hat{\mathcal{L}}_{\tau_L,h_L}(\theta_x, \theta_t) &= \left(
      \hat{A}_{\tau_L,h_L}(\theta_x)\right)^{-1} \left[
      K_{\tau_L} + h_L^{-2} \beta(\theta_x) M_{\tau_L} - e^{-\i \theta_t}
      N_{\tau_L} \right]\\
    &= I_{N_t} - e^{-\i \theta_t} \left[
      K_{\tau_L} + h_L^{-2} \beta(\theta_x) M_{\tau_L} \right]^{-1} N_{\tau_L}.
  \end{align*}
  Hence we have
  \begin{align*}
    \hat{\mathcal{S}}_{\tau_L,h_L}( \theta_x, \theta_t) &=  I_{N_t} - \omega_t \left(
    I_{N_t} -  e^{-\i \theta_t} \left(
      K_{\tau_L} + h_L^{-2} \beta(\theta_x) M_{\tau_L} \right)^{-1} N_{\tau_L}
  \right)\\
  &= (1 - \omega_t) I_{N_t} + \omega_t e^{-\i \theta_t} \left( K_{\tau_L} +
    h_L^{-2} \beta(\theta_x) M_{\tau_L} \right)^{-1} N_{\tau_L}.
  \end{align*}
  By induction this completes the proof.
\end{proof}

In view of Lemma \ref{chap4_lemma18}, the following mapping property holds
when periodic boundary conditions are assumed:
\begin{equation}\label{chap4_mappingProps2}
\begin{aligned}
  \mathcal{S}_{\tau_L,h_L}^\nu: \Psi_{L_x,L_t}(\theta_x,\theta_t) &\rightarrow
  \Psi_{L_x,L_t}(\theta_x,\theta_t),\\
  U &\mapsto (\hat{\mathcal{S}}_{\tau_L,h_L}( \theta_x, \theta_t))^\nu U.
\end{aligned}
\end{equation}

Next we will analyze the smoothing behavior for the high
frequencies. To do so, we consider two coarsening strategies: semi
coarsening in time, and full space-time coarsening. 
\begin{definition}[High and low frequency ranges]
  Let $N_{L_t}, N_{L_x} \in \IN$. We define the set of frequencies
  \begin{align*}
    \Theta_{L_x,L_t} := \left\{
      (\frac{2k\pi}{N_{L_x}},\frac{2\ell\pi}{N_{L_t}}): k =
      1-\frac{N_{L_x}}{2},\ldots,\frac{N_{L_x}}{2} \text{ and } \ell =
      1-\frac{N_{L_t}}{2},\ldots,\frac{N_{L_t}}{2} \right\} \subset (-\pi,\pi]^2,
  \end{align*}
  and the sets of low and high frequencies with respect to semi coarsening
  in time,
  \begin{align*}
  \Theta_{L_x,L_t}^{\mathrm{low},\mathrm{s}} := \Theta_{L_x,L_t} \cap
  (-\pi,\pi]\times(-\frac{\pi}{2},\frac{\pi}{2}],\quad
  \Theta_{L_x,L_t}^{\mathrm{high},\mathrm{s}} := \Theta_{L_x,L_t} \setminus
  \Theta_{L_x,L_t}^{\mathrm{low},\mathrm{s}},
  \end{align*}
  and full space-time coarsening 
  \begin{align*}
    \Theta_{L_x,L_t}^{\mathrm{low},\mathrm{f}} := \Theta_{L_x,L_t} \cap
  (-\frac{\pi}{2},\frac{\pi}{2}]^2,\quad
  \Theta_{L_x,L_t}^{\mathrm{high},\mathrm{f}} := \Theta_{L_x,L_t} \setminus
  \Theta_{L_x,L_t}^{\mathrm{low},\mathrm{f}}.
  \end{align*}
\end{definition}
\begin{figure}
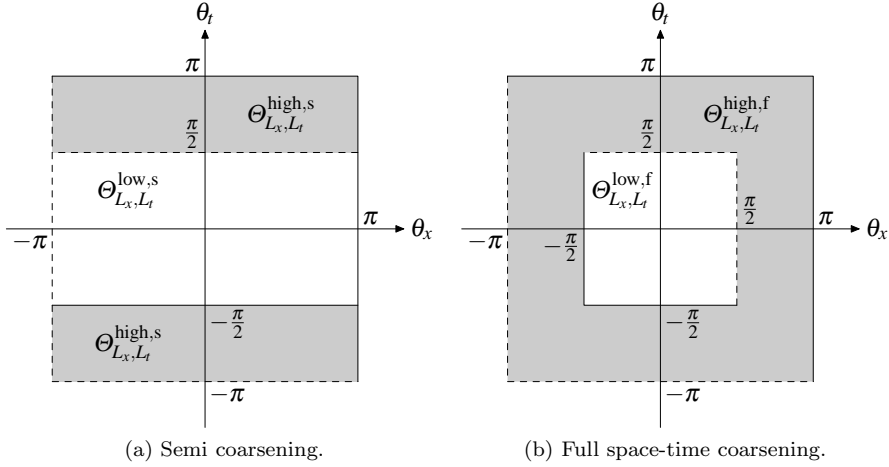

   \centering
    \subfloat[Semi coarsening.]
    {
      \scalebox{0.8}{\includegraphics{./figures/LowAndHighFrequencies.1}}
      \label{chap4figureLowAndHighFrequencies:a}
    }
    \subfloat[Full space-time coarsening.]
    {
      \scalebox{0.8}{\includegraphics{./figures/LowAndHighFrequencies.2}}
      \label{chap4figureLowAndHighFrequencies:b}
    }
    \caption{Low and high frequencies $\theta_x$ and $\theta_t$ for semi
      coarsening and full space-time coarsening.}
    \label{chap4_figureFrequencies}
\end{figure}
In Figure \ref{chap4_figureFrequencies}, the high and low frequencies are
illustrated for the two coarsening strategies.
\begin{definition}[Asymptotic smoothing factors]
  Let $\hat{\mathcal{S}}_{\tau_L,h_L}(\theta_x, \theta_t)$ be the symbol of the
  block Jacobi smoother. Then the smoothing factor for
  semi-coarsening in time is
  \begin{align*}
    \mu_S^\mathrm{s} := \max \left\{
      \varrho(\hat{\mathcal{S}}_{\tau_L,h_L}(\theta_x, \theta_t)) :
      (\theta_x,\theta_t) \in \Theta_{L_x,L_t}^{\mathrm{high},\mathrm{s}} \right\},
  \end{align*}
  and the smoothing factor for full space-time coarsening is
  \begin{align*}
    \mu_S^\mathrm{f} := \max \left\{
      \varrho(\hat{\mathcal{S}}_{\tau_L,h_L}(\theta_x, \theta_t)) :
      (\theta_x,\theta_t) \in \Theta_{L_x,L_t}^{\mathrm{high},\mathrm{f}} \right\}.
  \end{align*}
\end{definition}
To study the smoothing behavior, we need 
the eigenvalues of the Fourier symbol $\hat{\mathcal{S}}_{\tau_L,h_L}(\theta_x,
\theta_t)$:

\begin{lemma}\label{chap4_lemma20}
  The spectral radius of the Fourier symbol
  $\hat{\mathcal{S}}_{\tau_L,h_L}(\theta_x, \theta_t)$ is given by
  \begin{align*}
    \rho\left(  \hat{\mathcal{S}}_{\tau_L,h_L}(\theta_x, \theta_t)
    \right) = \max\left\{ \abs{1 - \omega_t},
        \hat{\mathcal{S}}(\omega_t,\alpha(\theta_x, \mu), \theta_t)
    \right\}
  \end{align*}
  with
  \begin{align*}
    \left(\hat{\mathcal{S}}(\omega_t,\alpha, \theta_t)\right)^2 := (1-
    \omega_t)^2 + 2 \omega_t (1-\omega_t) \alpha
    \cos(\theta_t) + \alpha^2 \omega_t^2,
  \end{align*}
  where $\alpha(\theta_x, \mu) := R(-\mu \beta(\theta_x))$ and $R(z)$ is the
  $(p_t,p_t+1)$ subdiagonal Pad\'{e} approximation of the exponential function
  $e^z$ and  $\mu := \tau_L h_L^{-2}$ is a discretization parameter.
\end{lemma}
\begin{proof}
   The eigenvalues of the Fourier symbol
     \begin{align*}
    \hat{\mathcal{S}}_{\tau_L,h_L}(\theta_x, \theta_t) = (1-\omega_t) I_{N_t}
    + \omega_t e^{-\i \theta_t} \left( K_{\tau_L} + h_L^{-2} \beta(\theta_x)
      M_{\tau_L} \right)^{-1} N_{\tau_L}
  \end{align*}
   are given by
   \begin{align*}
     \sigma(\hat{\mathcal{S}}_{\tau_L,h_L}(\theta_x, \theta_t)) = 1 - \omega_t + e^{-\i \theta_k} \omega_t  \sigma((K_{\tau_L} + h_L^{-2} \beta(\theta_x) M_{\tau_L})^{-1} N_{\tau_L}).
   \end{align*}
   With Lemma \ref{chap4_theorem2} and using the definition of $\alpha(\theta_x, \mu)$ we are now able to compute the spectrum as
   \begin{align*}
     \sigma(\hat{\mathcal{S}}_{\tau_L,h_L}(\theta_x, \theta_t)) = \left\{1 - \omega_t, 1 - \omega_t + e^{-\i \theta_k} \omega_t \alpha(\theta_x, \mu) \right\}.
   \end{align*}
   Hence we obtain the spectral radius
   \[ \varrho(\hat{\mathcal{S}}_{\tau_L,h_L}(\theta_x, \theta_t)) = \max \left\{ \abs{1 - \omega_t}, \abs{1 - \omega_t +
       e^{-\i \theta_k} \omega_t \alpha(\theta_x, \mu)} \right\}. \]
   Direct calculations lead to
   \begin{align*}
     \abs{1 - \omega_t +
       e^{-\i \theta_k} \omega_t \alpha(\theta_x, \mu)}^2 = (1-\omega_t)^2 + 2
   \omega_t (1 - \omega_t) \alpha(\theta_x, \mu) \cos(\theta_k) + (\alpha(\theta_x, \mu))^2
   \omega_t^2,
   \end{align*}
   which completes the proof.
\end{proof}

Next, we study the smoothing behavior of the damped block Jacobi
iteration for the case when semi-coarsening in time is
applied.
\begin{lemma}\label{chap4_lemma21}
  For the function
  \begin{align*}
    \left(\hat{\mathcal{S}}(\omega_t,\alpha, \theta_t)\right)^2 := (1-
    \omega_t)^2 + 2 \omega_t (1-\omega_t) \alpha
    \cos(\theta_t) + \alpha^2 \omega_t^2
  \end{align*}
 with $\alpha = \alpha(\theta_x, \mu)$ as defined in Lemma \ref{chap4_lemma20}
 and even polynomial degrees $p_t$, the min-max principle
\begin{align*}
  \inf_{\omega_t \in (0,1]} \sup_{\stackrel{\theta_t \in [\frac{\pi}{2},
      \pi]}{\theta_x \in [0,\pi] }}
  \hat{\mathcal{S}}(\omega_t,\alpha(\theta_x,\mu), \theta_t) =
  \frac{1}{\sqrt{2}}
\end{align*}
holds for any discretization parameter $\mu \geq 0$ with the optimal parameters
\begin{align*}
  \omega_t^\ast = \frac{1}{2}, \qquad \theta_t^\ast = \frac{\pi}{2} \qquad
  \text{and} \qquad \theta_x^\ast = 0.
\end{align*}
\end{lemma}
\begin{proof}
  Since we consider even polynomial degrees $p_t$, the $(p_t,p_t+1)$
  subdiagonal Pad\'{e} approximation $R(z)$ of the exponential
  function $e^z$ is positive for all $z \leq 0$. Hence we also have
  that $\alpha(\theta_x, \mu) = R(-\mu \beta(\theta_x))$ is positive
  for all $\mu \geq 0$ and $\theta_x \in [0,\pi]$. Since $\omega_t \in
  (0,1]$, we obtain
  \begin{align*}
     \theta_t^\ast := \argsup_{\theta_t \in [\frac{\pi}{2},\pi]}
     \hat{\mathcal{S}}(\omega_t,\alpha(\theta_x,\mu), \theta_t) =
     \frac{\pi}{2}.
  \end{align*}
  Since $\alpha(0,\mu)=1$ and $\abs{\alpha(\theta_x,\mu)} \leq 1$ for all
  $\theta_x \in [0,\pi]$ and $\mu \geq 0$ we get
  \begin{align*}
   \theta_x^\ast := \argsup_{\theta_x \in [0,\pi]}
   \hat{\mathcal{S}}(\omega_t,\alpha(\theta_x,\mu), \theta^\ast) = 0.
  \end{align*}
  Hence we have to find the infimum of
  \begin{align*}
    \left(\hat{\mathcal{S}}(\omega_t,\alpha(\theta_x^\ast,\mu),
      \theta_t^\ast)\right)^2 = (1-\omega_t)^2 + \omega_t^2,
  \end{align*}
  which is obtained for $\omega_t^\ast = \frac{1}{2}$. This implies that 
  \[  \left(\hat{\mathcal{S}}(\omega_t^\ast,\alpha(\theta_x^\ast,\mu),
      \theta_t^\ast)\right)^2 = \frac{1}{2}, \]
  which completes the proof.
\end{proof}
\begin{lemma}[Asymptotic smoothing factor for semi-coarsening]\label{chap4_lemma21b}
  For the function
    \begin{align*}
    \left(\hat{\mathcal{S}}(\omega_t,\alpha(\theta_x,\mu), \theta_t)\right)^2 = (1-
    \omega_t)^2 + 2 \omega_t (1-\omega_t) \alpha(\theta_x,\mu)
    \cos(\theta_t) + (\alpha(\theta_x,\mu))^2 \omega_t^2,
  \end{align*}
  where $\alpha = \alpha(\theta_x, \mu)$ is defined as in Lemma \ref{chap4_lemma20}
  and the choice $\omega_t^\ast = \frac{1}{2}$ and any polynomial degree $p_t \in
  \IN_0$, we have the bound
  \begin{align*}
  \sup_{\stackrel{\theta_t \in [\frac{\pi}{2},
      \pi]}{\theta_x \in [0,\pi] }}
  \hat{\mathcal{S}}(\omega_t,\alpha(\theta_x,\mu), \theta_t) \leq
  \frac{1}{\sqrt{2}}.
  \end{align*}
\end{lemma}
\begin{proof}
  For even polynomial degrees $p_t$, we can apply Lemma
  \ref{chap4_lemma21} to get the bound stated. For odd polynomial
  degrees, the $(p_t,p_t+1)$ subdiagonal Pad\'{e} approximation $R(z)$
  of the exponential function $e^z$ is negative for large negative
  values of $z$. If the value of $\alpha(\theta_x^\ast,\mu) = R(-\mu
  \beta(\theta_x^\ast))$ for the optimal parameter $\theta_x^\ast \in
       [0,\pi]$ is positive, we get directly the bound of Lemma
       \ref{chap4_lemma21b}. Otherwise we obtain
  \begin{align*}
     \theta_t^\ast := \argsup_{\theta_t \in [\frac{\pi}{2},\pi]}
     \hat{\mathcal{S}}(\omega_t,\alpha(\theta_x^\ast,\mu), \theta_t) = \pi.
  \end{align*}
  For a negative $\alpha(\theta_x^\ast,\mu)$, this implies that
  \begin{align*}
    \sup_{\stackrel{\theta_t \in [\frac{\pi}{2},
      \pi]}{\theta_x \in [0,\pi] }}
  \hat{\mathcal{S}}(\omega_t^\ast,\alpha(\theta_x,\mu),
  \theta_t) &\leq \frac{1}{2}(1 + \abs{\alpha(\theta_x^\ast,\mu)})
  \leq \frac{3}{4}(\sqrt{3}-1) < \frac{1}{\sqrt{2}},
  \end{align*}
  since any subdiagonal $(p_t,p_t+1)$ Pad\'{e} approximation $R(z)$ is
  bounded from below by $R(z) \geq \frac{1}{2}(5 - 3 \sqrt{3})$ for
  all $z < 0$.
\end{proof}

Lemma \ref{chap4_lemma21b} shows that the asymptotic smoothing factor
for semi-coarsening in time is bounded by $\mu_S^\mathrm{s} \leq
\frac{1}{\sqrt{2}}$. Hence, by applying the damped block Jacobi
smoother with the optimal damping parameter $\omega_t^\ast =
\frac{1}{2}$, the error components in the high
frequencies $\Theta_{L_x,L_t}^{\mathrm{high},\mathrm{s}}$ are
asymptotically damped by a factor of at least $\frac{1}{\sqrt{2}}$. 

\begin{lemma}[Asymptotic smoothing factor for full space-time coarsening]\label{chap4_lemma22}
  For the optimal choice of the damping parameter $\omega_t^\ast = \frac{1}{2}$,
  we have
  \begin{align*}
    \sup_{\theta_t \in [0,\pi]} \hat{\mathcal{S}}(\omega_t^\ast,\alpha,
    \theta_t) = \frac{1}{2}(1 + \abs{\alpha})
  \end{align*}
  with the optimal parameter
  \begin{align*}
    \theta_t^\ast = \begin{cases} 0 & \alpha \geq 0,\\ \pi & \alpha >
      0. \end{cases}
  \end{align*}
\end{lemma}
\begin{proof}
  Let $\alpha \in \IR$. For the optimal damping parameter $\omega_t^\ast =
  \frac{1}{2}$ we have
  \begin{align*}
     \left(\hat{\mathcal{S}}(\omega_t^\ast,\alpha, \theta_t)\right)^2 =
     \frac{1}{4}\left( 1 +2 \alpha \cos(\theta_t) + \alpha^2 \right).
  \end{align*}
  First we study the case $\alpha \geq 0$, where we get
  \begin{align*}
    \theta_t^\ast := \argsup_{\theta_t \in [0,\pi]}
    \hat{\mathcal{S}}(\omega_t^\ast,\alpha, \theta_t) = 0.
  \end{align*}
  For the case $\alpha < 0$ we  obtain
  \begin{align*}
    \theta_t^\ast := \argsup_{\theta_t \in [0,\pi]}
    \hat{\mathcal{S}}(\omega_t^\ast,\alpha, \theta_t) = \pi.
  \end{align*}
  This implies that
  \begin{align*}
     \left(\hat{\mathcal{S}}(\omega_t^\ast,\alpha, \theta_t\ast)\right)^2 =
     \frac{1}{4}\left( 1 +2 \abs{\alpha} + \alpha^2 \right) = \frac{1}{4}\left(
        1 + \abs{\alpha}\right)^2,
  \end{align*}
  which completes the proof.
\end{proof}

Lemma \ref{chap4_lemma22} shows that we obtain good smoothing
behavior for the high frequencies with respect to the space
discretization, i.e. $\theta_x \in \Theta_{L_x}^{\mathrm{high}}$, if
$\alpha = \alpha(\theta_x,\mu)$ is sufficiently small for any
frequency $\theta_x \in[\frac{\pi}{2},\pi]$.  Hence combining Lemma
\ref{chap4_lemma21b} with Lemma \ref{chap4_lemma22}, we see that good
smoothing behavior can be obtained for all frequencies
$(\theta_x,\theta_t) \in \Theta_{L_x,L_t}^{\mathrm{high},\mathrm{f}}$,
if the function $\alpha = \alpha(\theta_x,\mu)$ is sufficiently
small. This results in a restriction on the discretization parameter
$\mu$. With the next lemma we will analyze the behavior of the
smoothing factor $\mu_S^\mathrm{f}$ with respect to the discretization
parameter $\mu$ for even polynomial degrees $p_t \in \IN_0$.
\begin{lemma}\label{chap4_lemma23}
  Let $p_t \in \IN_0$ be even.  Then for the optimal choice of the damping
  parameter $\omega_t^\ast = \frac{1}{2}$ we have
  \begin{align*}
     \sup_{\stackrel{\theta_t \in [0,\pi]}{\theta_x \in [\frac{\pi}{2},\pi]}}
     \hat{\mathcal{S}}(\omega_t^\ast,\alpha(\theta_x,\mu),\theta_t) = \frac{1}{2}(1 +
     R(-3 \mu)),
  \end{align*}
  where $R(z)$ is the $(p_t,p_t+1)$ subdiagonal Pad\'{e} approximation of the
  exponential function $e^z$.
\end{lemma}
\begin{proof}
  In view of Lemma \ref{chap4_lemma22} it remains to compute the supremum
  \begin{align*}
    \sup_{\theta_x \in [\frac{\pi}{2},\pi]} \frac{1}{2}(1+\abs{\alpha(\theta_x,
      \mu)}).
  \end{align*}
  Since for even polynomial degrees $p_t$ the function $\alpha(\theta_x,\mu) =
  R(-\mu \beta(\theta_x))$ is monotonically decreasing in
  $\beta(\theta_x)$, the supremum is obtained for $\beta(\theta_x) = 3$, since
  $\beta(\theta_x) \in [3,12]$ for $\theta_x \in [\frac{\pi}{2},\pi]$. This
  implies that $\theta_x^\ast = \frac{\pi}{2}$, and we obtain the statement of
  the lemma with
  \begin{align*}
    \sup_{\stackrel{\theta_t \in [0,\pi]}{\theta_x \in [\frac{\pi}{2},\pi]}}
     \hat{\mathcal{S}}(\omega_t^\ast,\alpha(\theta_x,\mu)),\theta_t) =
     \hat{\mathcal{S}}(\omega_t^\ast,\alpha(\theta_x^\ast, \mu),\theta_t^\ast)
     = \frac{1}{2}(1+\abs{\alpha(\theta_x^\ast, \mu)})
     = \frac{1}{2}(1 + R(-3 \mu)).
  \end{align*}
\end{proof}

The proof of Lemma \ref{chap4_lemma23} only holds for even polynomial
degrees, but the result is also true for odd polynomial degrees $p_t$,
only the proof gets more involved, since the Pad\'{e} approximation
$R(z)$, $z\leq0$ is not monotonically decreasing for odd polynomial
degrees.
\begin{remark}\label{chap4_remarkCriticalDiscretizationParameter}
  In view of Lemma \ref{chap4_lemma23} we obtain a good smoothing
  behavior for the high frequencies in space $\theta_x \in
  \Theta_{L_x}^{\mathrm{high}}$, i.e. $\mu_S^{\mathrm{f}} \leq
  \frac{1}{\sqrt{2}}$, if the discretization parameter $\mu$ is large
  enough, i.e.
  \begin{align}\label{chap4_equationOptDiscrParam}
    \mu \geq \mu_{p_t}^\ast \quad \text{with} \quad R(-3\mu_{p_t}^\ast) = \sqrt{2} -1.
  \end{align}
  Hence we are able to compute the critical discretization parameter
  $\mu_{p_t}^\ast$ with respect to the polynomial degree $p_t$,
  \begin{equation*}
  \begin{aligned}
    \mu_0^\ast &= \frac{\sqrt{2}}{3} \approx 0.4714045208,\\
    \mu_1^\ast &= \frac{1}{3} (-3 - \sqrt{2} + \sqrt{11 + 12 \sqrt{2}}) \approx 0.2915022565,\\
    \mu_2^\ast &\approx 0.2938105446,\\
    \mu_3^\ast &\approx 0.2937911168,\\
    \mu_\infty^\ast &\approx 0.2937911957.
  \end{aligned}
  \end{equation*}
  To compute the critical discretization parameter $\mu_{\infty}^\ast$,
  we used the fact that the $(p_t,p_t+1)$ subdiagonal Pad\'{e}
  approximation $R(z)$ converges to the exponential function $e^z$ for
  $z \leq 0$ as $p_t \rightarrow \infty$.
\end{remark}
\begin{figure}[t]
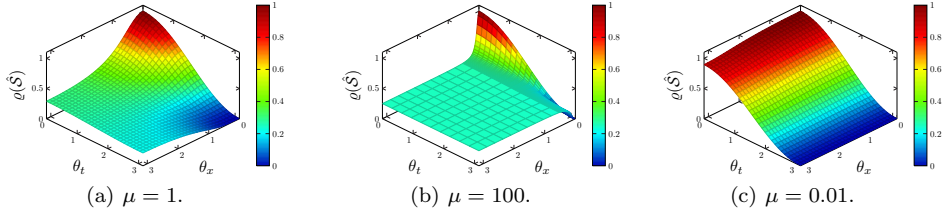

    \centering
    \subfloat[$\mu = 1$.]
    {
      \scalebox{0.34}{\input{./figures/plots/smoothing1.tex}}
      \label{chap4PlotSmoothingFactor1}
    }\hfill
    \subfloat[$\mu = 100$.]
    {
      \scalebox{0.34}{\input{./figures/plots/smoothing2.tex}}
      \label{chap4PlotSmoothingFactor2}
    }
    \subfloat[$\mu = 0.01$.]
    {
      \scalebox{0.34}{\input{./figures/plots/smoothing3.tex}}
      \label{chap4PlotSmoothingFactor3}
    }
    \caption{Smoothing factor
      $\hat{\mathcal{S}}(\omega_t^\ast,\alpha(\theta_x,
      \mu),\theta_t)^2$ for $\theta_x, \theta_t \in[0,\pi]$ for $p_t =
      0$ and different discretization parameters $\mu$.}
\end{figure}
\begin{remark}\label{remarkDampingParameter}
Lemma \ref{chap4_lemma22} shows that for all frequencies
$(\theta_x,\theta_t) \in \Theta_{L_x,L_t}$ we have the bound
\begin{align*}
    \hat{\mathcal{S}}(\omega_t^\ast,\alpha(\theta_x,\mu)),\theta_t)
    \leq \frac{1}{2}\left( 1 + \abs{R(-\beta(\theta_x) \mu)} \right)
    \leq 1.
\end{align*}
Only for $\theta_x = 0$ we have that $\beta(\theta_x)=0$, which
implies $R(-\beta(\theta_x) \mu)) = 1$. Hence if the discretization
parameter $\mu = \tau_L h_L^{-2}$ is large enough we have that
\[ 
  \abs{R(-\beta(\theta_x) \mu)}\approx 0
\]
for almost all frequencies $\theta_x \in \Theta_{L_x}$, which implies
a good smoothing behavior for almost all frequencies, see Figures
\ref{chap4PlotSmoothingFactor1}--\ref{chap4PlotSmoothingFactor3}. Only
the frequencies $\theta_x \in \Theta_{L_x}$ which are close to zero
imply $\hat{\mathcal{S}}(\omega_t^\ast,\alpha(\theta_x,\mu)),\theta_t)
\approx 1$. Hence for a large discretization parameter $\mu$ the
smoother itself is a good iterative solver for most frequencies, only
the frequencies $\theta_x \in \Theta_{L_x}$ which are close to zero,
i.e. very few low frequencies $\theta_x \in
\Theta_{L_x}^{\mathrm{low}}$, do not converge well. To obtain also a
perfect solver for a large discretization parameter $\mu$ we can
simply apply a correction step after one damped block Jacobi iteration
by restricting the defect in space several times until we arrive at a
very coarse problem. For this small problem one can solve the coarse
correction exactly by solving these small problems forward in
time. Afterward, we correct the solution by prolongating the coarse
corrections back to the fine space-grids.
\end{remark}

\subsection{Two-grid analysis}\label{chap4_twoGridAnalysis}

The iteration matrices for the two-grid cycles with
semi-coarsening and full space-time coarsening are 
\begin{align*}
  \mathcal{M}_{\tau_L,h_L}^{\mathrm{s}} &:= \mathcal{S}_{\tau_L,h_L}^{\nu_2} \left[I -
    \mathcal{P}_{\mathrm{s}}^{L_x,L_t}
    \left(\mathcal{L}_{2\tau_L,h_L}\right)^{-1}
    \mathcal{R}_{\mathrm{s}}^{L_x,L_t}\mathcal{L}_{\tau_L,h_L}
  \right]\mathcal{S}_{\tau_L,h_L}^{\nu_1},\\
  \mathcal{M}_{\tau_L,h_L}^{\mathrm{f}} &:= \mathcal{S}_{\tau_L,h_L}^{\nu_2} \left[I -
    \mathcal{P}_{\mathrm{f}}^{L_x,L_t}
    \left(\mathcal{L}_{2\tau_L,2h_L}\right)^{-1}
    \mathcal{R}_{\mathrm{f}}^{L_x,L_t}\mathcal{L}_{\tau_L,h_L}
  \right]\mathcal{S}_{\tau_L,h_L}^{\nu_1},
\end{align*}
with the restriction and prolongation matrices
\begin{align*}
  \mathcal{R}_{\mathrm{s}}^{L_x,L_t} := I_{N_{L_x}} \otimes \mathcal{R}^{L_t}
  ,\qquad\qquad & \mathcal{R}_{\mathrm{f}}^{L_x,L_t} := \mathcal{R}_x^{L_x}
  \otimes\mathcal{R}^{L_t},\\ 
  \mathcal{P}_{\mathrm{s}}^{L_x,L_t} := I_{N_{L_x}} \otimes \mathcal{P}^{L_t}
  ,\qquad\qquad & \mathcal{P}_{\mathrm{f}}^{L_x,L_t} := \mathcal{P}_x^{L_x}
  \otimes\mathcal{P}^{L_t}.
\end{align*}
The restriction and prolongation matrices in time,
i.e. $\mathcal{R}^{L_t}$ and $\mathcal{P}^{L_t}$ are given by (see
\cite{Gander:2014:AOA})
\begin{align}\label{chap4_restrictionTime}
  \mathcal{R}^L &:= \begin{pmatrix}
    R_1 & R_2 &     &        &        & \\
        &     & R_1 & R_2    &        & \\
        &     &     & \ddots & \ddots & \\
        &     &     &        & R_1    & R_2
  \end{pmatrix} \in \IR^{N_t N_L \times N_t N_{L-1}},\quad
  \mathcal{P}^L &:= (\mathcal{R}^L)^\top,
\end{align}
with the local prolongation matrices $R_1^\top := M_{\tau_L}^{-1}
\widetilde{M}_{\tau_L}^1$ and $R_2^\top := M_{\tau_L}^{-1}
\widetilde{M}_{\tau_L}^2$, where for basis functions $\left\{ \psi_k
\right\}_{k=1}^{N_t} \subset \IP^{p_t}(0,\tau_L)$ and $\left\{
\widetilde{\psi}_k \right\}_{k=1}^{N_t} \subset \IP^{p_t}(0, 2\tau_L)$
the local projection matrices from coarse to fine grids are defined
for $k,l = 1,\ldots,N_t$ by
\begin{align*}
  \widetilde{M}_{\tau_L}^1[k,\ell] := \int_{0}^{\tau_L}
  \widetilde{\psi}_\ell(t) \psi_k(t)
  \mathrm{d}t \quad \text{and} \quad \widetilde{M}_{\tau_L}^2[k,\ell] :=
  \int_{\tau_L}^{2 \tau_L} \widetilde{\psi}_\ell(t) \psi_k(t+\tau) \mathrm{d}t.
\end{align*}
The restriction and prolongation matrices in space for the one
dimensional case are
\begin{align}\label{chap4_spaceRestrictionProlongation}
  \mathcal{R}_x^{L_x} &:= \frac{1}{2}
  \begin{pmatrix}
    2 & 1 &   &   &   &   &\\
      & 1 & 2 & 1 &   &   &\\
      &   & \ddots & \ddots & \ddots & & \\
      &   &        & 1 & 2 & 1 & \\
      &   &        &   &   & 1 & 2
  \end{pmatrix} \in \IR^{N_{L_x} \times N_{L_x-1}},\\ \mathcal{P}_x^{L_x}
  &:= (\mathcal{R}_x^{L_x})^\top  \in \IR^{N_{L_x-1} \times N_{L_x}}.
\end{align}
To analyze the two-grid iteration matrices
$\mathcal{M}_{\tau_L,h_L}^{\mathrm{s}}$ and
$\mathcal{M}_{\tau_L,h_L}^{\mathrm{f}}$ we need
\begin{lemma}\label{chap4_lemma24}
   Let $\vec{u} = (\vec{u}_1,\vec{u}_2,\ldots,\vec{u}_{N_{L_t}})^\top \in \IR^{N_t
    N_{L_x} N_{L_t}}$ for $N_t, N_{L_x}, N_{L_t} \in \IN$ where we assume that
  $N_{L_x}$ and $N_{L_t}$
  are even numbers, and assume that
  \[ \vec{u}_n \in \IR^{N_{t} N_{L_x}} \qquad \text{and} \qquad \vec{u}_{n,r} \in
  \IR^{N_t} \]
  for $n=1,\ldots,N_{L_t}$ and $r=1,\ldots,N_{L_x}$.
  Then the vector $\vec{u}$ can be written as
  \begin{align*}
    \vec{u} = \sum_{(\theta_x,\theta_t) \in \Theta_{L_x,L_t}^{\mathrm{low,f}} } &\big[
    \vec{\psi}^{L_x,L_t}(\theta_x,\theta_t) +
    \vec{\psi}^{L_x,L_t}(\gamma(\theta_x),\theta_t) \\
    &\quad +\vec{\psi}^{L_x,L_t}(\theta_x,\gamma(\theta_t)) +
    \vec{\psi}^{L_x,L_t}(\gamma(\theta_x),\gamma(\theta_t)) \big],
  \end{align*}
  with the shifting operator $\gamma(\theta) := \theta -
  \mathrm{sign}(\theta) \pi$ and the vector
  $\vec{\psi}^{L_x,L_t}(\theta_x,\theta_t) \in \IC^{N_t N_{L_x}
    N_{L_t}}$ as in Lemma \ref{chap4_lemma15}.
\end{lemma}
\begin{proof}
  Using Lemma \ref{chap4_lemma15} and Lemma \ref{chap4_lemma9} leads
  to the desired result with
  \begin{align*}
    \vec{u} &= \sum_{\theta_x \in \Theta_{L_x} }\sum_{\theta_t \in \Theta_{L_t}}
    \vec{\psi}^{L_x,L_t}(\theta_x,\theta_t)\\
    &= \sum_{\theta_x \in \Theta_{L_x}^{\mathrm{low}} }\sum_{\theta_t \in
      \Theta_{L_t}^{\mathrm{low}}}
    \vec{\psi}^{L_x,L_t}(\theta_x,\theta_t) + \sum_{\theta_x \in
      \Theta_{L_x}^{\mathrm{high}} }\sum_{\theta_t \in
      \Theta_{L_t}^{\mathrm{low}}} \vec{\psi}^{L_x,L_t}(\theta_x,\theta_t)\\
    &\quad + \sum_{\theta_x \in \Theta_{L_x}^{\mathrm{low}} }\sum_{\theta_t \in
      \Theta_{L_t}^{\mathrm{high}}}
    \vec{\psi}^{L_x,L_t}(\theta_x,\theta_t) + \sum_{\theta_x \in
      \Theta_{L_x}^{\mathrm{high}} }\sum_{\theta_t \in
      \Theta_{L_t}^{\mathrm{high}}} \vec{\psi}^{L_x,L_t}(\theta_x,\theta_t)\\
    &=\sum_{(\theta_x,\theta_t) \in \Theta_{L_x,L_t}^{\mathrm{low,f}} } \big[
    \vec{\psi}^{L_x,L_t}(\theta_x,\theta_t) +
    \vec{\psi}^{L_x,L_t}(\gamma(\theta_x),\theta_t) \\
    &\qquad\qquad\qquad\qquad +\vec{\psi}^{L_x,L_t}(\theta_x,\gamma(\theta_t)) +
    \vec{\psi}^{L_x,L_t}(\gamma(\theta_x),\gamma(\theta_t)) \big].
  \end{align*}
\end{proof}

\begin{definition}[Space of harmonics]
  For $N_t, N_{L_x}, N_{L_t} \in \IN$ and the frequencies $(\theta_x,\theta_t)
  \in \Theta_{L_x,L_t}^{\mathrm{low,f}}$ let the vector
  $\vec{\Phi}^{L_x,L_t}(\theta_x,\theta_t) \in \IC^{N_t N_{L_x} N_{L_t}}$ be
  as in Lemma \ref{chap4_lemma15}. Then we define the linear space of
  harmonics with frequencies $(\theta_x,\theta_t)$ as
  \begin{align*}
    \mathcal{E}_{L_x,L_t}(\theta_x,\theta_t) &:= \mathrm{span} \big\{
      \vec{\Phi}^{L_x,L_t}(\theta_x,\theta_t),
      \vec{\Phi}^{L_x,L_t}(\gamma(\theta_x),\theta_t),\\
      &\qquad\qquad\vec{\Phi}^{L_x,L_t}(\theta_x,\gamma(\theta_t)),
      \vec{\Phi}^{L_x,L_t}(\gamma(\theta_x),\gamma(\theta_t)) \big\}\\
      &\phantom{:}= \big\{ \vec{\psi}^{L_x,L_t}(\theta_x,\theta_t) \in \IC^{N_t
        N_{L_x} N_{L_t}}:\\
      &\quad\quad \vec{\psi}_{n,r}^{L_x,L_t}(\theta_x,\theta_t) =
      U_1\vec{\Phi}_{n,r}^{L_x,L_t}(\theta_x,\theta_t) +
      U_2\vec{\Phi}_{n,r}^{L_x,L_t}(\gamma(\theta_x),\theta_t)\\
      &\quad\qquad\qquad\qquad\qquad+
      U_3\vec{\Phi}_{n,r}^{L_x,L_t}(\theta_x,\gamma(\theta_t)) +
      U_4\vec{\Phi}_{n,r}^{L_x,L_t}(\gamma(\theta_x),\gamma(\theta_t)), \\
      &\quad\quad \text{for all } n=1,\ldots,N_{L_t}, r =
      1,\ldots,N_{L_x} \text{ and } U_1,U_2,U_3,U_4 \in \IC^{N_t \times N_t}\big\}.
  \end{align*}
\end{definition}
With the assumption of periodic boundary conditions, see
(\ref{chap4_periodicBCST}), Lemma \ref{chap4_lemma17} implies for the
system matrix $\mathcal{L}_{\tau_L,h_L}$ for all frequencies
$(\theta_x,\theta_t) \in \Theta_{L_x,L_t}^{\mathrm{low,f}}$ the
  mapping property:
\begin{equation}\label{chap4_mappingL}
\begin{aligned}
  \mathcal{L}_{\tau_L,h_L}: \mathcal{E}_{L_x,L_t}(\theta_x,\theta_t)
  &\rightarrow \mathcal{E}_{L_x,L_t}(\theta_x,\theta_t)\\
  \begin{pmatrix} U_1\\U_2\\U_3\\U_4 \end{pmatrix} &\mapsto
    \left(\begin{array}{l}
    \hat{\mathcal{L}}_{\tau_L,h_L}(\theta_x, \theta_t) U_1 \\
    \hat{\mathcal{L}}_{\tau_L,h_L}(\gamma(\theta_x), \theta_t) U_2 \\
    \hat{\mathcal{L}}_{\tau_L,h_L}(\theta_x, \gamma(\theta_t)) U_3 \\
    \hat{\mathcal{L}}_{\tau_L,h_L}(\gamma(\theta_x), \gamma(\theta_t)) U_4
    \end{array}\right)
   =:
  \widetilde{\mathcal{L}}_{\tau_L,h_L}(\theta_x,\theta_t) \begin{pmatrix}
    U_1\\U_2\\U_3\\U_4 \end{pmatrix},
\end{aligned}
\end{equation}
where $\widetilde{\mathcal{L}}_{\tau_L,h_L}(\theta_x,\theta_t) \in
\IC^{4 N_t \times 4 N_t}$ is a block diagonal matrix. With the same
arguments, we obtain with Lemma \ref{chap4_lemma18} for the smoother
for all frequencies $(\theta_x,\theta_t) \in
\Theta_{L_x,L_t}^{\mathrm{low,f}}$ the mapping property
\begin{equation}\label{chap4_mappingS}
\begin{aligned}
  \mathcal{S}_{\tau_L,h_L}^\nu: \mathcal{E}_{L_x,L_t}(\theta_x,\theta_t)
  &\rightarrow \mathcal{E}_{L_x,L_t}(\theta_x,\theta_t)\\
  \begin{pmatrix} U_1\\U_2\\U_3\\U_4 \end{pmatrix} &\mapsto
  \left(\begin{array}{l}
    (\hat{\mathcal{S}}_{\tau_L,h_L}(\theta_x, \theta_t))^\nu U_1 \\
    (\hat{\mathcal{S}}_{\tau_L,h_L}(\gamma(\theta_x), \theta_t))^\nu U_2 \\
    (\hat{\mathcal{S}}_{\tau_L,h_L}(\theta_x, \gamma(\theta_t)))^\nu U_3 \\
    (\hat{\mathcal{S}}_{\tau_L,h_L}(\gamma(\theta_x), \gamma(\theta_t)))^\nu U_4
   \end{array}\right) =:
  \left(\widetilde{\mathcal{S}}_{\tau_L,h_L}(\theta_x,\theta_t)\right)^\nu
  \begin{pmatrix} U_1\\U_2\\U_3\\U_4 \end{pmatrix},
\end{aligned}
\end{equation}
with the block diagonal matrix
$\widetilde{\mathcal{S}}_{\tau_L,h_L}(\theta_x,\theta_t) \in \IC^{4 N_t
  \times 4 N_t}$. 

To analyze the two-grid cycle on the space of harmonics
$\mathcal{E}_{L_x,L_t}(\theta_x,\theta_t)$ for frequencies $(\theta_x,\theta_t)
\in \Theta_{L_x,L_t}^{\mathrm{low,f}}$, we further have to investigate the
mapping properties of the restriction and prolongation operators for the two
different coarsening strategies $\mathcal{R}_{\mathrm{s}}^{L_x,L_t},
\mathcal{R}_{\mathrm{f}}^{L_x,L_t}$ and $\mathcal{P}_{\mathrm{s}}^{L_x,L_t},
\mathcal{P}_{\mathrm{f}}^{L_x,L_t}$.
\begin{lemma}\label{chap4_lemma25}
  Let $\mathcal{R}_x^{L_x}$ and $\mathcal{P}_x^{L_x}$ be the restriction and
  prolongation matrices as defined in
  (\ref{chap4_spaceRestrictionProlongation}). For $\theta_x \in
  \Theta_{L_x}^{\mathrm{low}}$ let
  $\vec{\varphi}^{L_x}(\theta_x) \in \IC^{N_{L_x}}$ and
  $\vec{\varphi}^{L_x-1}(2\theta_x) \in \IC^{N_{L_x-1}}$ be defined as in Theorem
  \ref{chap4_theorem1}. Then 
  \begin{align*}
    \left(\mathcal{R}_x^{L_x} \vec{\varphi}^{L_x}(\theta_x)\right)[r] =
    \hat{\mathcal{R}}_x(\theta_x) \vec{\varphi}^{L_x-1}(2 \theta_x) [r],
  \end{align*}
  for $r = 2,\ldots,N_{L_x-1}-1$ with the Fourier symbol
  $\hat{\mathcal{R}}_x(\theta_x) := 1 + \cos(\theta_x)$. For the prolongation
  operator we further have
  \begin{align*}
    \left( \mathcal{P}_x^{L_x} \vec{\varphi}^{L_x-1}(2 \theta_x)\right)[s] =
    \left( \hat{\mathcal{P}}_x(\theta_x)\vec{\varphi}^{L_x}(\theta_x) +
      \hat{\mathcal{P}}_x(\gamma(\theta_x))\vec{\varphi}^{L_x}(\gamma(\theta_x))
    \right)[s],
  \end{align*}
  for $s = 2,\ldots,N_{L_x}-1$ with the Fourier symbol
  $\hat{\mathcal{P}}_x(\theta_x) := \frac{1}{2}\hat{\mathcal{R}}_x(\theta_x)$.
\end{lemma}
\begin{proof}
  The prove is classical, see \cite{Trottenberg2001} for example.
\end{proof}

\begin{definition}
   For $N_t, N_{L_x}, N_{L_t} \in \IN$ and the frequencies $(\theta_x,\theta_t)
  \in \Theta_{L_x,L_t}^{\mathrm{low,f}}$ let the vector
  $\vec{\Phi}^{L_x,L_t-1}(\theta_x,\theta_t) \in \IC^{N_t N_{L_x} N_{L_t-1}}$ be
  defined as in Lemma \ref{chap4_lemma15}. Then we define the linear space with
  frequencies $(\theta_x,2\theta_t)$ as
  \begin{align*}
    \Psi_{L_x,L_t-1}(\theta_x,2\theta_t) :&= \mathrm{span}\left\{
      \vec{\Phi}^{L_x,L_t-1}(\theta_x,2\theta_t),
      \vec{\Phi}^{L_x,L_t-1}(\gamma(\theta_x),2\theta_t) \right\}\\
    &= \big\{ \vec{\psi}^{L_x,L_t-1}(\theta_x,2 \theta_t)  \in \IC^{N_t N_{L_x}
      N_{L_t-1}} :\\
      &\quad\quad \vec{\psi}_{n,r}^{L_x,L_t-1}(\theta_x,2 \theta_t) = U_1
      \vec{\Phi}_{n,r}^{L_x,L_t-1}(\theta_x,2\theta_t) + U_2
      \vec{\Phi}_{n,r}^{L_x,L_t-1}(\gamma(\theta_x),2\theta_t)\\
      &\qquad\qquad \text{for all } n = 1,\ldots,N_{L_t}, r = 1,\ldots, N_{L_x}
      \text{ and } U_1, U_2 \in \IC^{N_t \times N_t} \big\}.
  \end{align*}
\end{definition}

For semi-coarsening, the next lemma shows the mapping
property for the restriction operator $\mathcal{R}_{\mathrm{s}}^{L_x,L_t}$.
\begin{lemma}\label{chap4_lemma26}
 The restriction operator $\mathcal{R}_{\mathrm{s}}^{L_x,L_t}$ satisfies the
 mapping property
  \begin{align*}
    \mathcal{R}_{\mathrm{s}}^{L_x,L_t} :
    \mathcal{E}_{L_x,L_t}(\theta_x,\theta_t) \rightarrow
    \Psi_{L_x,L_t-1}(\theta_x,2\theta_t)
  \end{align*}
  with the mapping
  \begin{align*}
    \begin{pmatrix} U_1\\U_2\\U_3\\U_4 \end{pmatrix} \mapsto
      \widetilde{\mathcal{R}}_{\mathrm{s}}(\theta_t) &\begin{pmatrix}
      U_1\\U_2\\U_3\\U_4 \end{pmatrix},
  \end{align*}
  and the matrix
  \begin{align*}
    \widetilde{\mathcal{R}}_{\mathrm{s}}(\theta_t) := \begin{pmatrix}
      \hat{\mathcal{R}}(\theta_t) & 0 &
      \hat{\mathcal{R}}(\gamma(\theta_t)) & 0 \\
      0 & \hat{\mathcal{R}}(\theta_t) & 0 &
      \hat{\mathcal{R}}(\gamma(\theta_t)) 
    \end{pmatrix}
  \end{align*}
  with the Fourier symbol $\hat{\mathcal{R}}(\theta_t) \in \IC^{N_t \times N_t}$
  as defined in Lemma \ref{chap4_lemma12}.
\end{lemma}
\begin{proof}
  Let $\vec{\Phi}^{L_x,L_t}(\theta_x, \theta_t) \in \Psi_{L_x, L_t}(\theta_x,
  \theta_t)$ and $\vec{\Phi}^{L_x,L_t-1}(\theta_x, 2\theta_t) \in \Psi_{L_x,
    L_t-1}(\theta_x, 2\theta_t)$ be defined as in Lemma
  \ref{chap4_lemma15}. Then for $n=1,\ldots,N_{L_t-1}$ and $r =
  1,\ldots,N_{L_x}$ we have,
  using Lemma \ref{chap4_lemma12},
  \begin{align*}
    \left( \mathcal{R}_{\mathrm{s}}^{L_x,L_t} \vec{\Phi}^{L_x,L_t}(\theta_x,
      \theta_t) \right)_{n,r} &= \sum_{s=1}^{N_{L_x}}\sum_{m=1}^{N_{L_t}}
    I_{N_{L_x}}[r,s] \mathcal{R}^{L_t}[n,m]\vec{\Phi}_{m,s}^{L_x,L_t}(\theta_x,
      \theta_t)\\
      &= \vec{\varphi}^{L_x}(\theta_x)[r] \sum_{m=1}^{N_{L_t}}
      \mathcal{R}^{L_t}[n,m]\vec{\Phi}_{m}^{L_t}(\theta_t)\\
      &= \vec{\varphi}^{L_x}(\theta_x)[r] \left(
        \mathcal{R}^{L_t}\vec{\Phi}^{L_t}(\theta_t) \right)_n\\
      &= \hat{\mathcal{R}}(\theta_t) \vec{\Phi}_{n}^{L_t-1}(2\theta_t)
      \vec{\varphi}^{L_x}(\theta_x)[r]\\
      &= \hat{\mathcal{R}}(\theta_t)
      \vec{\Phi}_{n,r}^{L_x,L_t-1}(\theta_x,2\theta_t).
  \end{align*}
  Applying this result to the vector $\vec{\psi}^{L_x,L_t}(\theta_x,\theta_t)
  \in
  \mathcal{E}_{L_x,L_t}(\theta_x,\theta_t)$ with $(\theta_x,\theta_t) \in
  \Theta_{L_x,L_t}^{\mathrm{f}}$ results in
  \begin{align*}
    \left( \mathcal{R}_{\mathrm{s}}^{L_x,L_t} \vec{\psi}^{L_x,L_t}(\theta_x,
      \theta_t) \right)_{n,r} &= \hat{\mathcal{R}}(\theta_t) U_1
    \vec{\Phi}_{n,r}^{L_x,L_t-1}(\theta_x,2\theta_t)\\
    &\quad+ \hat{\mathcal{R}}(\theta_t) U_2
    \vec{\Phi}_{n,r}^{L_x,L_t-1}(\gamma(\theta_x),2\theta_t)\\
    &\quad+ \hat{\mathcal{R}}(\gamma(\theta_t)) U_3
    \vec{\Phi}_{n,r}^{L_x,L_t-1}(\theta_x,2\gamma(\theta_t))\\
    &\quad+ \hat{\mathcal{R}}(\gamma(\theta_t)) U_4
    \vec{\Phi}_{n,r}^{L_x,L_t-1}(\gamma(\theta_x),2\gamma(\theta_t)).
    \intertext{Since $\vec{\Phi}_{n,r}^{L_x,L_t-1}(\theta_x,2\gamma(\theta_t))
      = \vec{\Phi}_{n,r}^{L_x,L_t-1}(\theta_x,2\theta_t)$ we further obtain}
    &= \left[  \hat{\mathcal{R}}(\theta_t) U_1 +
      \hat{\mathcal{R}}(\gamma(\theta_t)) U_3 \right]
    \vec{\Phi}_{n,r}^{L_x,L_t-1}(\theta_x,2\theta_t)\\
    &\quad+ \left[  \hat{\mathcal{R}}(\theta_t) U_2 +
      \hat{\mathcal{R}}(\gamma(\theta_t)) U_4
    \right]\vec{\Phi}_{n,r}^{L_x,L_t-1}(\gamma(\theta_x),2\theta_t),
  \end{align*}
  which completes the proof.
\end{proof}

\begin{lemma}\label{chap4_lemma27}
  With the assumptions of periodic boundary conditions
  (\ref{chap4_periodicBCST}) the following mapping
  property for the restriction operator holds:
  \begin{align*}
    \mathcal{R}_{\mathrm{f}}^{L_x,L_t} :
    \mathcal{E}_{L_x,L_t}(\theta_x,\theta_t) \rightarrow
    \Psi_{L_x-1,L_t-1}(2\theta_x,2\theta_t)
  \end{align*}
  with the mapping
  \begin{align*}
    \begin{pmatrix} U_1\\U_2\\U_3\\U_4 \end{pmatrix} \mapsto
      \widetilde{\mathcal{R}}_{\mathrm{f}}(\theta_x,\theta_t) &\begin{pmatrix}
      U_1\\U_2\\U_3\\U_4 \end{pmatrix}
  \end{align*}
  and the matrix
  \begin{align*}
    \widetilde{\mathcal{R}}_{\mathrm{f}}(\theta_x,\theta_t) := \begin{pmatrix}
      \hat{\mathcal{R}}(\theta_x,\theta_t) &
      \hat{\mathcal{R}}(\gamma(\theta_x),\theta_t) &
      \hat{\mathcal{R}}(\theta_x,\gamma(\theta_t)) &
      \hat{\mathcal{R}}(\gamma(\theta_x),\gamma(\theta_t))      
    \end{pmatrix}
  \end{align*}
  with the Fourier symbol
  \begin{align*}
    \hat{\mathcal{R}}(\theta_x,\theta_t) :=
    \hat{\mathcal{R}}_x(\theta_x)\hat{\mathcal{R}}(\theta_t) \in
    \IC^{N_t \times N_t},
  \end{align*}
  where $\hat{\mathcal{R}}_x(\theta_x) \in \IC$ is defined as in Lemma
  \ref{chap4_lemma25}.
\end{lemma}
\begin{proof}
  For the frequencies $(\theta_x,\theta_t) \in \Theta_{L_x,L_t}^{\mathrm{low}}$
  let $\vec{\Phi}^{L_x,L_t}(\theta_x, \theta_t) \in \Psi_{L_x, L_t}(\theta_x,
  \theta_t)$ and $\vec{\Phi}^{L_x-1,L_t-1}(2\theta_x, 2\theta_t) \in \Psi_{L_x-1,
    L_t-1}(2\theta_x, 2\theta_t)$ be defined as in Lemma
  \ref{chap4_lemma15}. Then for $n=1,\ldots,N_{L_t-1}$ and $r =
  2,\ldots,N_{L_x-1}-1$ we have
  \begin{align*}
    \left( \mathcal{R}_{\mathrm{f}}^{L_x,L_t} \vec{\Phi}^{L_x,L_t}(\theta_x,
      \theta_t) \right)_{n,r} &= \sum_{s=1}^{N_{L_x}}\sum_{m=1}^{N_{L_t}}
    \mathcal{R}_{x}^{L_x}[r,s]
    \mathcal{R}^{L_t}[n,m]\vec{\Phi}_{m,s}^{L_x,L_t}(\theta_x,\theta_t)\\
      &= \left(\sum_{s=1}^{N_{L_x}}\mathcal{R}_{x}^{L_x}[r,s]
        \vec{\varphi}^{L_x}(\theta_x)[r]\right)\left( \sum_{m=1}^{N_{L_t}}
      \mathcal{R}^{L_t}[n,m]\vec{\Phi}_{m}^{L_t}(\theta_t)\right)\\
      &= \left(\mathcal{R}_x^{L_x}\vec{\varphi}^{L_x}(\theta_x)\right)[r] \left(
        \mathcal{R}^{L_t}\vec{\Phi}^{L_t}(\theta_t) \right)_n.\\
      \intertext{Applying Lemma \ref{chap4_lemma25} and Lemma
        \ref{chap4_lemma12} leads to}
      &= \hat{\mathcal{R}}_x(\theta_x)\hat{\mathcal{R}}(\theta_t)
      \vec{\Phi}_{n}^{L_t-1}(2\theta_t)
      \vec{\varphi}^{L_x-1}(2\theta_x)[r]\\
      &= \hat{\mathcal{R}}(\theta_x, \theta_t)
      \vec{\Phi}_{n,r}^{L_x-1,L_t-1}(2\theta_x,2\theta_t).
  \end{align*}
   Using this result for the vector $\vec{\psi}^{L_x,L_t}(\theta_x,\theta_t)
  \in
  \mathcal{E}_{L_x,L_t}(\theta_x,\theta_t)$ with $(\theta_x,\theta_t) \in
  \Theta_{L_x,L_t}^{\mathrm{f}}$ results in
  \begin{align*}
    \left( \mathcal{R}_{\mathrm{f}}^{L_x,L_t} \vec{\psi}^{L_x,L_t}(\theta_x,
      \theta_t) \right)_{n,r} &= \hat{\mathcal{R}}(\theta_x,\theta_t) U_1
    \vec{\Phi}_{n,r}^{L_x-1,L_t-1}(2\theta_x,2\theta_t)\\
    &\quad+ \hat{\mathcal{R}}(\gamma(\theta_x),\theta_t) U_2
    \vec{\Phi}_{n,r}^{L_x-1,L_t-1}(2\gamma(\theta_x),2\theta_t)\\
    &\quad+ \hat{\mathcal{R}}(\theta_x,\gamma(\theta_t)) U_3
    \vec{\Phi}_{n,r}^{L_x-1,L_t-1}(2\theta_x,2\gamma(\theta_t))\\
    &\quad+ \hat{\mathcal{R}}(\gamma(\theta_x),\gamma(\theta_t)) U_4
    \vec{\Phi}_{n,r}^{L_x-1,L_t-1}(2\gamma(\theta_x),2\gamma(\theta_t)).
  \end{align*}
  With the relations
  \begin{align*}
    \vec{\Phi}^{L_x-1,L_t-1}(2\theta_x,2\theta_t)
      &= \vec{\Phi}^{L_x-1,L_t-1}(2\gamma(\theta_x),2\theta_t) 
      = \vec{\Phi}_{n,r}^{L_x-1,L_t-1}(2\theta_x,2\gamma(\theta_t))\\
      &= \vec{\Phi}_{n,r}^{L_x-1,L_t-1}(2\gamma(\theta_x),2\gamma(\theta_t)),
  \end{align*}
  we obtain the statement of this lemma with
  \begin{align*}
    \left( \mathcal{R}_{\mathrm{f}}^{L_x,L_t} \vec{\psi}^{L_x,L_t}(\theta_x,
      \theta_t) \right)_{n,r} &= \big[  \hat{\mathcal{R}}(\theta_x,\theta_t)
    U_1
    +
       \hat{\mathcal{R}}(\gamma(\theta_x),\theta_t) U_2 
    + \hat{\mathcal{R}}(\theta_x,\gamma(\theta_t)) U_3\\
    &\qquad+ \hat{\mathcal{R}}(\gamma(\theta_x),\gamma(\theta_t)) U_4
    \big]\vec{\Phi}_{n,r}^{L_x-1,L_t-1}(2\theta_x,2\theta_t).
  \end{align*}
\end{proof}

With the next lemmas, we will analyze the mapping properties of the
prolongation operators $\mathcal{P}_{\mathrm{s}}^{L_x,L_t}$ and
$\mathcal{P}_{\mathrm{f}}^{L_x,L_t}$.
\begin{lemma}\label{chap4_lemma28}
  For $(\theta_x,\theta_t) \in \Theta_{L_x,L_t}^{\mathrm{f}}$ the prolongation
  operator $\mathcal{P}_{\mathrm{s}}^{L_x,L_t}$ satisfies the mapping
  property
  \begin{align*}
    \mathcal{P}_{\mathrm{s}}^{L_x,L_t} :
    \Psi_{L_x,L_t-1}(\theta_x,2\theta_t) \rightarrow
    \mathcal{E}_{L_x,L_t}(\theta_x,\theta_t)
  \end{align*}
  with the mapping
  \begin{align*}
    \begin{pmatrix} U_1\\U_2 \end{pmatrix} \mapsto
      \begin{pmatrix}
      \hat{\mathcal{P}}(\theta_t) & 0\\
      0 &\hat{\mathcal{P}}(\theta_t) \\
      \hat{\mathcal{P}}(\gamma(\theta_t)) & 0 \\
      0 & \hat{\mathcal{P}}(\gamma(\theta_t)) 
    \end{pmatrix} \begin{pmatrix}
      U_1\\U_2 \end{pmatrix}
    =:\widetilde{\mathcal{P}}_{\mathrm{s}}(\theta_t) &\begin{pmatrix}
      U_1\\U_2 \end{pmatrix}
  \end{align*}
  and the Fourier symbol $\hat{\mathcal{P}}(\theta_t) \in \IC^{N_t \times N_t}$
  defined as in Lemma \ref{chap4_lemma13}.
\end{lemma}
\begin{proof}
  Let $\vec{\psi}^{L_x,L_t-1}(\theta_x,2\theta_t) \in
  \vec{\Psi}_{L_x,L_t-1}(\theta_x,2\theta_t)$ for $(\theta_x,\theta_t) \in
  \Theta_{L_x,L_t}^{\mathrm{f}}$. Then we have for $n=1,\ldots,N_{L_t}$ and $r
  = 1,\ldots, N_{L_x}$ that
  \begin{align*}
    \left( \mathcal{P}_{\mathrm{s}}^{L_x,L_t}
      \vec{\psi}^{L_x,L_t-1}(\theta_x,2\theta_t) \right)_{n,r} &=
    \sum_{s=1}^{N_{L_x}}\sum_{m=1}^{N_{L_t}}
    I_{N_{L_x}}[r,s]\mathcal{P}^{L_t}[n,m]
    \vec{\psi}_{m,s}^{L_x,L_t-1}(\theta_x,2\theta_t)\\
    &= \sum_{m=1}^{N_{L_t}}\mathcal{P}^{L_t}[n,m]
    \vec{\psi}_{m,r}^{L_x,L_t-1}(\theta_x,2\theta_t)\\
    &= \sum_{m=1}^{N_{L_t}}\mathcal{P}^{L_t}[n,m]
    \Big[(\vec{\varphi}^{L_x}(\theta_x)[r] U_1)
    \vec{\Phi}_{m}^{L_t-1}(2\theta_t) \\
    &\qquad\qquad\qquad\qquad+ (\vec{\varphi}^{L_x}(\gamma(\theta_x))[r] U_2)
    \vec{\Phi}_{m}^{L_t-1}(2\theta_t) \Big].
    \intertext{Since $\vec{\varphi}^{L_x}(\theta_x)[r] U_1 \in \IC^{N_t\times
        N_t}$ and $\vec{\varphi}^{L_x}(\gamma(\theta_x))[r] U_2 \in
      \IC^{N_t\times N_t}$ we further obtain by applying Lemma
      \ref{chap4_lemma13} that}
    &= \hat{\mathcal{P}}(\theta_t) (\vec{\varphi}^{L_x}(\theta_x)[r] U_1)
    \vec{\Phi}_{n}^{L_t}(\theta_t)\\
    &\quad+ \hat{\mathcal{P}}(\gamma(\theta_t))
    (\vec{\varphi}^{L_x}(\theta_x)[r]
    U_1)\vec{\Phi}_{n}^{L_t}(\gamma(\theta_t))\\
    &\quad+ \hat{\mathcal{P}}(\theta_t)
    (\vec{\varphi}^{L_x}(\gamma(\theta_x))[r] U_2)
    \vec{\Phi}_{n}^{L_t}(\theta_t)\\
    &\quad+  \hat{\mathcal{P}}(\gamma(\theta_t))
    (\vec{\varphi}^{L_x}(\gamma(\theta_x))[r]
    U_2)\vec{\Phi}_{n}^{L_t}(\gamma(\theta_t)).\\
    \intertext{With the definition of the Fourier mode
      $\vec{\Phi}_{n,r}^{L_x,L_t}(\theta_x,\theta_t)$ we get}
    &= \hat{\mathcal{P}}(\theta_t) U_1
    \vec{\Phi}_{n,r}^{L_x,L_t}(\theta_x,\theta_t) \\
    &\quad+ \hat{\mathcal{P}}(\gamma(\theta_t))
    U_1\vec{\Phi}_{n,r}^{L_x,L_t}(\theta_x,\gamma(\theta_t))\\
    &\quad+ \hat{\mathcal{P}}(\theta_t) U_2
    \vec{\Phi}_{n,r}^{L_x,L_t}(\gamma(\theta_x),\theta_t) \\
    &\quad+ \hat{\mathcal{P}}(\gamma(\theta_t))
    U_2\vec{\Phi}_{n,r}^{L_x,L_t}(\gamma(\theta_x),\gamma(\theta_t)),
  \end{align*}
  which completes the proof.
\end{proof}

\begin{lemma}\label{chap4_lemma29}
   With the assumptions of periodic boundary conditions
   (\ref{chap4_periodicBCST}) the following mapping
  property for the prolongation operator holds:
  \begin{align*}
    \mathcal{P}_{\mathrm{f}}^{L_x,L_t} :
    \Psi_{L_x-1,L_t-1}(2\theta_x,2\theta_t) \rightarrow
    \mathcal{E}_{L_x,L_t}(\theta_x,\theta_t)
  \end{align*}
  with the mapping
  \begin{align*}
     U \mapsto
      \left(\begin{array}{l}
      \hat{\mathcal{P}}(\theta_x,\theta_t)\\
      \hat{\mathcal{P}}(\gamma(\theta_x),\theta_t) \\
      \hat{\mathcal{P}}(\theta_x,\gamma(\theta_t)) \\
      \hat{\mathcal{P}}(\gamma(\theta_x),\gamma(\theta_t)) 
    \end{array}\right) U
    =:\widetilde{\mathcal{P}}_{\mathrm{f}}(\theta_x,\theta_t) U
  \end{align*}
  and the Fourier symbol
  \begin{align*}
    \hat{\mathcal{P}}(\theta_x,\theta_t) := \hat{\mathcal{P}}_x(\theta_x)
    \hat{\mathcal{P}}(\theta_t),
  \end{align*}
  where $\hat{\mathcal{P}}(\theta_t)$ is defined as in Lemma
  \ref{chap4_lemma13}.
\end{lemma}
\begin{proof}
   Let $\vec{\psi}^{L_x-1,L_t-1}(2\theta_x,2\theta_t) \in
  \vec{\Psi}_{L_x-1,L_t-1}(2\theta_x,2\theta_t)$ for $(\theta_x,\theta_t) \in
  \Theta_{L_x,L_t}^{\mathrm{f}}$. Then we have for $n=1,\ldots,N_{L_t}$ and $r
  = 2,\ldots, N_{L_x}-1$ that
  \begin{align*}
    &\left( \mathcal{P}_{\mathrm{f}}^{L_x,L_t}
      \vec{\psi}^{L_x-1,L_t-1}(2\theta_x,2\theta_t) \right)_{n,r} =
    \sum_{s=1}^{N_{L_x}}\sum_{m=1}^{N_{L_t}}
    \mathcal{P}_{x}^{L_x}[r,s]\mathcal{P}^{L_t}[n,m]
    \vec{\psi}_{m,s}^{L_x-1,L_t-1}(2\theta_x,2\theta_t)\\
    &\qquad= \left( \sum_{s=1}^{N_{L_x}} \mathcal{P}_{x}^{L_x}[r,s]
      \vec{\varphi}^{L_x-1}(2 \theta_x)[s] \right)\left(
      \sum_{m=1}^{N_{L_t}}\mathcal{P}^{L_t}[n,m]\vec{\Phi}_{m}^{L_t-1}(2\theta_t)
    \right).\\
    \intertext{Using Lemma \ref{chap4_lemma13} gives}
    &\qquad=\left( \hat{\mathcal{P}}_x(\theta_x)
      \vec{\varphi}^{L_x}(\theta_x)[r]+ \hat{\mathcal{P}}_x(\gamma(\theta_x))
      \vec{\varphi}^{L_x}(\gamma(\theta_x))[r] \right)\\
    &\qquad\qquad\qquad\times \left(
      \hat{\mathcal{P}}(\theta_t) U \vec{\Phi}_{n}^{L_t}(\theta_t)+
      \hat{\mathcal{P}}(\gamma(\theta_t)) U
      \vec{\Phi}_{n}^{L_t}(\gamma(\theta_t)) \right).\\
    \intertext{Using now the definition of the Fourier mode
      $\vec{\Phi}_{n,r}^{L_x,L_t}(\theta_x,\theta_t)$ leads to}
    &\qquad=\hat{\mathcal{P}}(\theta_x,\theta_t) U
    \vec{\Phi}_{n,r}^{L_x,L_t}(\theta_x,\theta_t)
    +\hat{\mathcal{P}}(\gamma(\theta_x),\theta_t) U
    \vec{\Phi}_{n,r}^{L_x,L_t}(\gamma(\theta_x),\theta_t) \\
    &\qquad\qquad+ \hat{\mathcal{P}}(\theta_x,\gamma(\theta_t))
    U\vec{\Phi}_{n,r}^{L_x,L_t}(\theta_x,\gamma(\theta_t)) +
    \hat{\mathcal{P}}(\gamma(\theta_x), \gamma(\theta_t))
    U\vec{\Phi}_{n,r}^{L_x,L_t}(\gamma(\theta_x),\gamma(\theta_t)),
  \end{align*}
  which completes the proof.
\end{proof}

For periodic boundary conditions (\ref{chap4_periodicBCST}), we
further obtain with Lemma \ref{chap4_lemma14} the mapping property for
the coarse grid correction, when semi coarsening in time is applied,
\begin{equation}\label{chap4_mappinginvLs}
\begin{aligned}
  \left( \mathcal{L}_{2 \tau_L, h_L} \right)^{-1} : \Psi_{L_x, L_t-1}(\theta_x,
  2 \theta_t) &\rightarrow \Psi_{L_x, L_t-1}(\theta_x,
  2 \theta_t)\\
  \begin{pmatrix} U_1 \\ U_2  \end{pmatrix} &\mapsto
  \left( \widetilde{\mathcal{L}}_{2\tau_L,
      h_L}^{\mathrm{s}}(\theta_x,2\theta_t) \right)^{-1} \begin{pmatrix} U_1 \\
    U_2  \end{pmatrix} \in \IC^{2 N_t \times N_t}
\end{aligned}
\end{equation}
with the matrix
\begin{align*}
  \left( \widetilde{\mathcal{L}}_{2\tau_L,
      h_L}^{\mathrm{s}}(\theta_x,2\theta_t) \right)^{-1} := \begin{pmatrix} 
    \left(\hat{\mathcal{L}}_{2\tau_L, h_L}(\theta_x, 2\theta_t)\right)^{-1} & 0 \\ 0 &
      \left(\hat{\mathcal{L}}_{2\tau_L, h_L}(\gamma(\theta_x),\theta_t)\right)^{-1}
  \end{pmatrix} \in \IC^{2 N_t \times 2N_t}.
\end{align*}
For full space-time coarsening, we have the mapping property
\begin{equation}\label{chap4_mappinginvLf}
\begin{aligned}
  \left( \mathcal{L}_{2 \tau_L, 2h_L} \right)^{-1} : \Psi_{L_x-1, L_t-1}(2\theta_x,
  2 \theta_t) &\rightarrow \Psi_{L_x-1, L_t-1}(2\theta_x,
  2 \theta_t)\\
  U &\mapsto
  \left( \widetilde{\mathcal{L}}_{2\tau_L,
      2h_L}^{\mathrm{f}}(2\theta_x,2\theta_t) \right)^{-1} U \in \IC^{N_t \times
  N_t},
\end{aligned}
\end{equation}
with
$\left( \widetilde{\mathcal{L}}_{2\tau_L,
      2h_L}^{\mathrm{f}}(2\theta_x,2\theta_t) \right)^{-1} :=
  \left(\hat{\mathcal{L}}_{2\tau_L, 2h_L}(2\theta_x, 2\theta_t)\right)^{-1}
  \in \IC^{N_t \times N_t}$.
We can now prove the following two theorems:
\begin{theorem}\label{cahp4_theorem5}
  Let $(\theta_x,\theta_t) \in
  \Theta_{L_x,L_t}^{\mathrm{low},\mathrm{f}}$. With the assumption
  of periodic boundary conditions (\ref{chap4_periodicBCST}), the
  following mapping property holds for the two-grid operator
  $\mathcal{M}_{\tau_L,h_L}^{\mathrm{s}}$ with semi
  coarsening in time:
  \begin{align*}
    \mathcal{M}_{\tau_L, h_L}^{\mathrm{s}} : \mathcal{E}_{L_x,L_t}(\theta_x,
    \theta_t) \rightarrow \mathcal{E}_{L_x,L_t}(\theta_x, \theta_t),
  \end{align*}
  with the mapping
  \begin{align*}
    \begin{pmatrix} U_1\\U_2\\U_3\\U_4 \end{pmatrix} \mapsto
    \widetilde{\mathcal{M}}_\mu^{\mathrm{s}}(\theta_k,\theta_t) \begin{pmatrix}
      U_1\\U_2\\U_3\\U_4 \end{pmatrix}
  \end{align*}
  and the iteration matrix
  \begin{align*}
    \widetilde{\mathcal{M}}_\mu^{\mathrm{s}}(\theta_k,\theta_t) :=
    \left(\widetilde{\mathcal{S}}_{\tau_L,h_L}(\theta_x,\theta_t)\right)^{\nu_2}
    \widetilde{\mathcal{K}}_{\mathrm{s}}(\theta_x,\theta_t)
    \left(\widetilde{\mathcal{S}}_{\tau_L,h_L}(\theta_x,\theta_t)\right)^{\nu_1}
    \in \IC^{4 N_t \times 4 N_t}
  \end{align*}
  with
  \begin{align*}
    \widetilde{\mathcal{K}}_{\mathrm{s}}(\theta_x,\theta_t) := I_{4 N_t} -
    \widetilde{\mathcal{P}}_{\mathrm{s}}(\theta_t)  \left(
      \widetilde{\mathcal{L}}_{2\tau_L,h_L}^{\mathrm{s}}(\theta_x,2\theta_t)
    \right)^{-1} \widetilde{\mathcal{R}}_{\mathrm{s}}(\theta_t)
    \widetilde{\mathcal{L}}_{\tau_L,h_L}(\theta_x,\theta_t).
  \end{align*}
\end{theorem}
\begin{proof}
  The statement of this theorem follows by using Lemma \ref{chap4_lemma26},
  Lemma \ref{chap4_lemma28} and the mapping properties (\ref{chap4_mappingL}),
  (\ref{chap4_mappingS}) and (\ref{chap4_mappinginvLs}).
\end{proof}

\begin{theorem}\label{cahp4_theorem6}
  Let $(\theta_x,\theta_t) \in
  \Theta_{L_x,L_t}^{\mathrm{low},\mathrm{f}}$. With the
  assumption of periodic boundary conditions
  (\ref{chap4_periodicBCST}), the following mapping property
  holds for the two-grid operator
  $\mathcal{M}_{\tau_L,h_L}^{\mathrm{f}}$ with full space-time
  coarsening:
  \begin{align*}
    \mathcal{M}_{\tau_L, h_L}^{\mathrm{f}} : \mathcal{E}_{L_x,L_t}(\theta_x,
    \theta_t) \rightarrow \mathcal{E}_{L_x,L_t}(\theta_x, \theta_t),
  \end{align*}
  with the mapping
  \begin{align*}
    \begin{pmatrix} U_1\\U_2\\U_3\\U_4 \end{pmatrix} \mapsto
    \widetilde{\mathcal{M}}_\mu^{\mathrm{f}}(\theta_k,\theta_t) \begin{pmatrix}
      U_1\\U_2\\U_3\\U_4 \end{pmatrix}
  \end{align*}
  and the iteration matrix
  \begin{align*}
    \widetilde{\mathcal{M}}_\mu^{\mathrm{f}}(\theta_k,\theta_t) :=
    \left(\widetilde{\mathcal{S}}_{\tau_L,h_L}(\theta_x,\theta_t)\right)^{\nu_2}
    \widetilde{\mathcal{K}}_{\mathrm{f}}(\theta_x,\theta_t)
    \left(\widetilde{\mathcal{S}}_{\tau_L,h_L}(\theta_x,\theta_t)\right)^{\nu_1}
    \in \IC^{4 N_t \times 4 N_t}
  \end{align*}
  with
  \begin{align*}
    \widetilde{\mathcal{K}}_{\mathrm{f}}(\theta_x,\theta_t) := I_{4 N_t} -
    \widetilde{\mathcal{P}}_{\mathrm{f}}(\theta_x,\theta_t)  \left(
      \widetilde{\mathcal{L}}_{2\tau_L,2h_L}^{\mathrm{f}}(2\theta_x,2\theta_t)
    \right)^{-1} \widetilde{\mathcal{R}}_{\mathrm{f}}(\theta_x,\theta_t)
    \widetilde{\mathcal{L}}_{\tau_L,h_L}(\theta_x,\theta_t).
  \end{align*}
\end{theorem}
\begin{proof}
 The statement of this theorem is a direct consequence of Lemma
 \ref{chap4_lemma27},
  Lemma \ref{chap4_lemma29} and the mapping properties (\ref{chap4_mappingL}),
  (\ref{chap4_mappingS}) and (\ref{chap4_mappinginvLf}).
\end{proof}

In view of Lemma \ref{chap4_lemma24} we can now represent the initial error
$\vec{e}^0 = \vec{u}-\vec{u}^0$ as 
\begin{align*}
  \vec{e}^0 &= \sum_{(\theta_x,\theta_t) \in \Theta_{L_x,L_t}^{\mathrm{low,f}} } \big[
    \vec{\psi}^{L_x,L_t}(\theta_x,\theta_t) +
    \vec{\psi}^{L_x,L_t}(\gamma(\theta_x),\theta_t) \\
    &\qquad\qquad\qquad\qquad +\vec{\psi}^{L_x,L_t}(\theta_x,\gamma(\theta_t)) +
    \vec{\psi}^{L_x,L_t}(\gamma(\theta_x),\gamma(\theta_t)) \big] \\
    &=:\sum_{(\theta_x,\theta_t) \in \Theta_{L_x,L_t}^{\mathrm{low,f}} }
    \tilde{\vec{\psi}}(\theta_x,\theta_t),
\end{align*}
with $\tilde{\vec{\psi}}(\theta_x,\theta_t) \in
\mathcal{E}_{L_x,L_t}(\theta_x,\theta_t)$ for all $(\theta_x,\theta_t)
\in \Theta_{L_x,L_t}^{\mathrm{low,f}}$. Using Theorem
\ref{cahp4_theorem5} and Theorem \ref{cahp4_theorem6} we now can
analyze the asymptotic convergence behavior of the two-grid cycle by
simply computing the largest spectral radius of
$\widetilde{\mathcal{M}}_\mu^{\mathrm{s}}(\theta_k,\theta_t)$ or
$\widetilde{\mathcal{M}}_\mu^{\mathrm{f}}(\theta_k,\theta_t)$ with
respect to the frequencies $(\theta_x,\theta_t) \in
\Theta_{L_x,L_t}^{\mathrm{low,f}}$. This motivates
\begin{definition}[Asymptotic two-grid convergence factors]
  For the two-grid iteration matrices $\mathcal{M}_{\tau_L, h_L}^{\mathrm{s}}$
  and $\mathcal{M}_{\tau_L, h_L}^{\mathrm{f}}$ we define the asymptotic
  convergence factors
  \begin{align*}
    \varrho(\hat{\mathcal{M}}_\mu^\mathrm{s}) &:= \max \left\{
      \varrho(\widetilde{\mathcal{M}}_\mu^{\mathrm{s}}(\theta_k,\theta_t)):
      (\theta_x,\theta_t) \in \Theta_{L_x,L_t}^{\mathrm{low},\mathrm{f}} \text{
      with } \theta_x \neq 0\right\},\\
      \varrho(\hat{\mathcal{M}}_\mu^\mathrm{f}) &:= \max \left\{
      \varrho(\widetilde{\mathcal{M}}_\mu^{\mathrm{f}}(\theta_k,\theta_t)):
      (\theta_x,\theta_t) \in \Theta_{L_x,L_t}^{\mathrm{low},\mathrm{f}} \text{
      with } \theta_x \neq 0\right\}.
  \end{align*}
\end{definition}

Note that in the definition of the two-grid convergence factors we
have neglected all frequencies $(0,\theta_t) \in
\Theta_{L_x,L_t}^{\mathrm{low},\mathrm{f}}$, since the Fourier symbol
with respect to the Laplacian is zero for $\theta_x = 0$, see also the
remarks in \cite[chapter 4]{Trottenberg2001}.

To derive the assymptotic convergence factors
$\varrho(\hat{\mathcal{M}}_\mu^\mathrm{s})$ and $
\varrho(\hat{\mathcal{M}}_\mu^\mathrm{f})$ for a given discretization
parameter $\mu \in \IR_+$ and a given polynomial degree $p_t \in
\IN_0$ we have to compute the eigenvalues of
\begin{equation}\label{chap4_twogridIterationMatrices}
\widetilde{\mathcal{M}}_\mu^{\mathrm{s}}(\theta_k,\theta_t) \in \IC^{4N_t \times
4N_t} \quad \text{and} \quad
\widetilde{\mathcal{M}}_\mu^{\mathrm{f}}(\theta_k,\theta_t) \in
\IC^{4N_t \times 4N_t},
\end{equation}
with $N_t = p_t +1$ for each low frequency $(\theta_x,\theta_t) \in
\Theta_{L_x,L_t}^{\mathrm{low},\mathrm{f}}$. Since it is difficult to find 
closed form expressions for the eigenvalues of the iteration
matrices (\ref{chap4_twogridIterationMatrices}), we will compute
the eigenvalues numerically. In particular we will compute the average
convergence factors for the domain $\Omega = (0,1)$ with a decomposition into
$1024$ uniform sub intervals, i.e. $N_{L_x} = 1023$. Furthermore we will
analyze the two-grid cycles for $N_{L_t} = 256$ time steps.

We plot as solid lines in the Figures
\ref{chap4PlotSemiP0}--\ref{chap4PlotSemiP1} the theoretical
convergence factors $\varrho(\hat{\mathcal{M}}_\mu^\mathrm{s})$ as
functions of the discretization parameter $\mu = \tau_L h_L^{-2} \in
[10^{-6},10^6]$ for different polynomial degrees $p_t \in \left\{0,1
\right\}$, and different number of smoothing steps $\nu_1 = \nu_2 =
\nu \in \{1,2,5\}$. We observe that the theoretical convergence
factors are always bounded by
$\varrho(\hat{\mathcal{M}}_\mu^\mathrm{s})\leq \frac{1}{2}$. For the
case when semi coarsening in time is applied we see that the two-grid
cycle converges for any discretization parameter $\mu$. We also see
for polynomial degree $p_t = 1$ that the theoretical convergence
factors are much smaller than the theoretical convergence factors for
the lowest order case $p_t = 0$.

We also plot in Figures \ref{chap4PlotSemiP0}--\ref{chap4PlotSemiP1}
using dots, triangles and squares the numerically measured convergence
factors for solving the equation
\[ \mathcal{L}_{\tau_L,h_L} \vec{u} = \vec{f} \]
with the two-grid cycle when semi coarsening in time is applied. For
the numerical test we use a zero right hand side, i.e. $\vec{f} =
\mathbf{0}$ and a random initial vector $\vec{u}^0$ with values
between zero and one. The convergence factor of the two-grid cycle is
measured by
\begin{align*}
  \max_{k = 1,\ldots,N_{\mathrm{iter}}}
  \frac{\norm{\vec{r}^{k+1}}_2}{\norm{\vec{r}^{k}}_2}, \quad \text{with
  }\vec{r}^k := \vec{f} - \mathcal{L}_{\tau_L, h_L} \vec{u}^k,
\end{align*}
where $N_{\mathrm{iter}} \in \IN$, $N_{\mathrm{iter}} \leq 250$ is the
number of two-grid iterations used until we have reached a given
relative error reduction of $\varepsilon_{\mathrm{MG}} = 10^{-140}$.
We observe that the numerical results agree very well with
the theoretical results, even though the local Fourier mode analysis
is not rigorous for the numerical simulation that does not use
periodic boundary conditions.
\begin{figure}
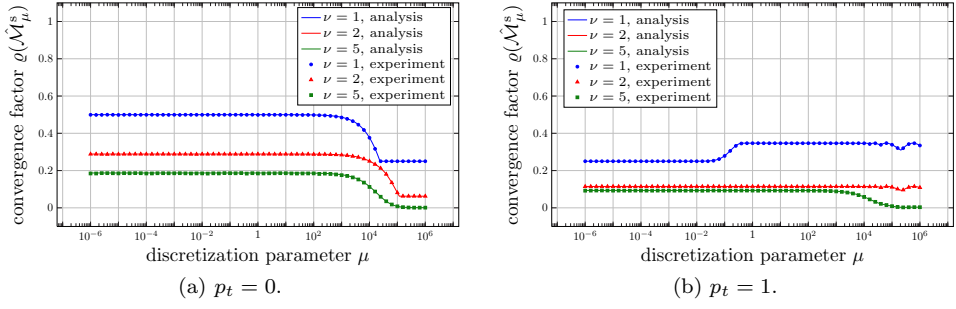

    \centering
    \subfloat[$p_t = 0$.]
    {
      \scalebox{0.43}{\input{./figures/plots/plotSTsemi0.tex}}
      \label{chap4PlotSemiP0}
    }\hfill
    \subfloat[$p_t = 1$.]
    {
      \scalebox{0.43}{\input{./figures/plots/plotSTsemi1.tex}}
      \label{chap4PlotSemiP1}
    }
    \caption{Asymptotic convergence factor
      $\varrho(\hat{\mathcal{M}}_{\mu}^\mathrm{s})$ for different
      discretization parameters $\mu$ and numerical convergence
      factors for $N_t = 256$ time steps and $N_x = 1023$.}
  \end{figure}

In Figures \ref{chap4PlotFullP0}--\ref{chap4PlotFullP1} we plot the theoretical
convergence factors $\varrho(\hat{\mathcal{M}}_\mu^\mathrm{s})$ for the
two-grid cycle $\mathcal{M}_{\tau_L,h_L}^{\mathrm{f}}$ with full
space-time coarsening as function of the discretization parameter
$\mu \in [10^{-6},10^6]$ for different polynomial degrees $p_t \in
\{0,1\}$. We observe that the theoretical convergence factors are bounded by
$\varrho(\hat{\mathcal{M}}_\mu^\mathrm{f})\leq \frac{1}{2}$ if the
discretization parameter $\mu$ is large enough, i.e. for $\mu \geq \mu^\ast$. In
Remark \ref{chap4_remarkCriticalDiscretizationParameter} we already
computed these critical values $\mu^\ast$ for several polynomial degrees
$p_t$. As before we compared the theoretical results with the numerical
results when full space-time coarsening is applied. In Figures
\ref{chap4PlotFullP0}--\ref{chap4PlotFullP1} the measured numerical convergence
factors are plotted as dots, triangles and squares. We see that the
theoretical results agree very well with the numerical results.

Overall we conclude that the two-grid cycle always converges to the
exact solution of the linear system (\ref{chap4_generalLinearSystem})
when semi coarsening in time is applied. Furthermore, if the
discretization parameter $\mu$ is large enough, we can also apply full
space-time coarsening, which leads to a smaller coarse problem
compared to the semi coarsening case.

  \begin{figure}
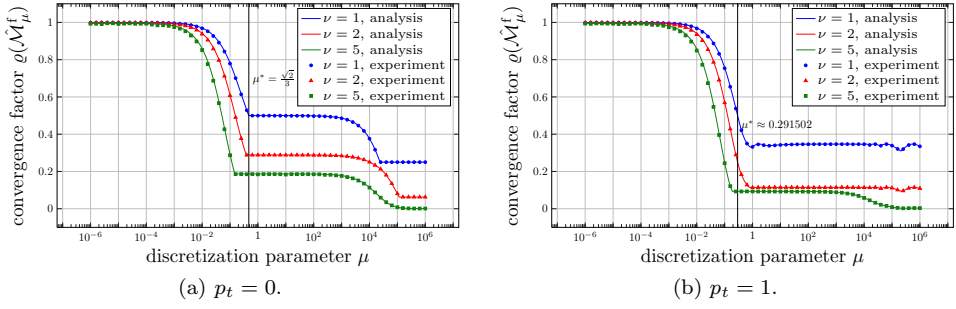

    \centering
    \subfloat[$p_t = 0$.]
    {
      \scalebox{0.43}{\input{./figures/plots/plotSTfull0.tex}}
      \label{chap4PlotFullP0}
    }\hfill
    \subfloat[$p_t = 1$.]
    {
      \scalebox{0.43}{\input{./figures/plots/plotSTfull1.tex}}
      \label{chap4PlotFullP1}
    }
    \caption{Average convergence factor
      $\varrho(\hat{\mathcal{M}}_{\mu}^\mathrm{f})$ for different
      discretization parameters $\mu$ and numerical convergence
      factors for $N_t = 256$ time steps and $N_x = 1023$.}
  \end{figure}

\begin{remark}\label{V-cycleApproxSpace}
  For the two-grid analysis above we used for the block Jacobi smoother 
  \begin{align}\label{chap4_smootherExact}
    \mathcal{S}_{\tau_L,h_L}^\nu = \left[I - \omega_t
    (D_{\tau_L,h_L})^{-1} \mathcal{L}_{\tau_L,h_L} \right]^\nu
  \end{align}
  the exact inverse of the diagonal matrix $D_{\tau_L,h_L} =
  \mathrm{diag}\left\{ A_{\tau_L,h_L} \right\}_{n=1}^{N_{L_t}}$. In
  practice, it is more efficient to use an approximation
  $\widetilde{D}_{\tau_L,h_L}^{-1}$ by applying one multigrid
  iteration in space for the blocks $A_{\tau_L,h_L}$, see also
    \cite{Wathen}, where such an approximate block Jacobi method is
    used directly to precondition GMRES. Hence the smoother
  (\ref{chap4_smootherExact}) changes to
  \begin{align}\label{chap4_smootherApprox}
    \overline{\mathcal{S}}_{\tau_L,h_L}^\nu := \left[I - \omega_t
    \left( I - \mathcal{M}_{\tau_L,h_L}  \right)(D_{\tau_L,h_L})^{-1}
    \mathcal{L}_{\tau_L,h_L} \right]^\nu,
  \end{align}
  with the matrix $\mathcal{M}_{\tau_L,h_L} := \mathrm{diag}\left\{
    \mathcal{M}_{\tau_L,h_L}^x\right\}_{n=1}^{N_{L_t}}$, where
  $\mathcal{M}_{\tau_L,h_L}^x$ is the iteration matrix of the
  multigrid scheme for the matrix $A_{\tau_L,h_L}$. In the case
  that the iteration matrix $\mathcal{M}_{\tau_L,h_L}^x$ is given by a two-grid
  cycle, we further obtain the representation
  \begin{align*}
    \mathcal{M}_{\tau_L,h_L}^x = \mathcal{S}_{\tau_L,h_L}^{x,\nu_2^x} \left[ I -
      \overline{\mathcal{P}}_x^{L_x} A_{\tau_L,2h_L}^{-1}
      \overline{\mathcal{R}}_x^{L_x}A_{\tau_L,h_L} 
    \right] \mathcal{S}_{\tau_L,h_L}^{x,\nu_1^x},
  \end{align*}
  with a damped Jacobi smoother in space 
  \begin{align*}
    \mathcal{S}_{\tau_L,h_L}^{x,\nu^x} := \left[I - \omega_x
    \left(D_{\tau_L,h_L}^x\right)^{-1} A_{\tau_L,h_L}\right]^{\nu^x}, \quad D_{\tau_L,h_L}^x
    := \mathrm{diag}\big\{ \frac{2h}{3} K_{\tau_L} + \frac{2}{h} M_{\tau_L}
    \big\}_{r=1}^{N_{L_x}}
  \end{align*}
  and the restriction and prolongation operators
  \begin{align*}
    \overline{\mathcal{R}}_x^{L_x} := \mathcal{R}_x^{L_x} \otimes I_{N_t}
    \quad \text{and} \quad \overline{\mathcal{P}}_x^{L_x} :=
    \mathcal{P}_x^{L_x} \otimes I_{N_t}.
  \end{align*}
  With the different smoother (\ref{chap4_smootherApprox}) we also have to
  analyze the two different two-grid iteration matrices
  \begin{align}\label{chap4_twoGridIterationMatrixApprox}
    \overline{\mathcal{M}}_{\tau_L,h_L}^{\mathrm{s}} &:=
    \overline{\mathcal{S}}_{\tau_L,h_L}^{\nu_2} \left[I -
      \mathcal{P}_{\mathrm{s}}^{L_x,L_t}
    \left(\mathcal{L}_{2\tau_L,h_L}\right)^{-1}
    \mathcal{R}_{\mathrm{s}}^{L_x,L_t}\mathcal{L}_{\tau_L,h_L}
  \right]\overline{\mathcal{S}}_{\tau_L,h_L}^{\nu_1},\\
  \overline{\mathcal{M}}_{\tau_L,h_L}^{\mathrm{f}} &:=
  \overline{\mathcal{S}}_{\tau_L,h_L}^{\nu_2} \left[I -
    \mathcal{P}_{\mathrm{f}}^{L_x,L_t}
    \left(\mathcal{L}_{2\tau_L,2h_L}\right)^{-1}
    \mathcal{R}_{\mathrm{f}}^{L_x,L_t}\mathcal{L}_{\tau_L,h_L}
  \right]\overline{\mathcal{S}}_{\tau_L,h_L}^{\nu_1}.
  \end{align}
  Hence it remains to analyze the mapping property of the operator
  $\mathcal{M}_{\tau_L,h_L}$ on the space of harmonics
  $\mathcal{E}_{L_x,L_t}(\theta_x,\theta_t)$. By several computations we find
  under the assumptions of periodic boundary conditions
  (\ref{chap4_periodicBCST}) that
  \begin{align*}
    \mathcal{M}_{\tau_L,h_L} : \mathcal{E}_{L_x,L_t}(\theta_x,\theta_t)
    \rightarrow \mathcal{E}_{L_x,L_t}(\theta_x,\theta_t)
  \end{align*}
  with the mapping
  \begin{align}\label{chap4_mappingPropertyApprox}
    \begin{pmatrix} U_1\\U_2\\U_3\\U_4 \end{pmatrix} \mapsto
    \begin{pmatrix}
      \widetilde{M}_{\tau_L,h_L}(\theta_x) & 0 \\
      0 & \widetilde{M}_{\tau_L,h_L}(\theta_x)
    \end{pmatrix}
    \begin{pmatrix} U_1\\U_2\\U_3\\U_4 \end{pmatrix},
  \end{align}
  and the iteration matrix
  \begin{align*}
    \widetilde{M}_{\tau_L,h_L}(\theta_x) &:=
    \widetilde{\mathcal{S}}_{\tau_L,h_L}^{x,\nu_1^x}(\theta_x)
    \mathcal{K}_{\tau_L,h_L}^x(\theta_x)
    \widetilde{\mathcal{S}}_{\tau_L,h_L}^{x,\nu_2^x}(\theta_x) \in \IC^{2 N_t
      \times 2N_t},\\
    \mathcal{K}_{\tau_L,h_L}^x(\theta_x,\theta_t)&:= I_{2N_t} -
    \widetilde{\mathcal{P}}_x(\theta_x)\widetilde{A}_{\tau_L,2h_L}^{-1}(2\theta_x)
    \widetilde{\mathcal{R}}_x(\theta_x) \widetilde{A}_{\tau_L,h_L}(\theta_x)\in
    \IC^{2 N_t \times 2N_t},
    \end{align*}
    with the matrices
    \begin{align*}
    \widetilde{A}_{\tau_L,h_L}(\theta_x) &:=
    \begin{pmatrix} 
      \hat{A}_{\tau_L,h_L}(\theta_x) & 0\\
      0 & \hat{A}_{\tau_L,h_L}(\gamma(\theta_x))
    \end{pmatrix} \in \IC^{2 N_t \times 2N_t},\\
    \widetilde{A}_{\tau_L,2h_L}^{-1}(2\theta_x) &:= \left(
      \hat{A}_{\tau_L,2h_L}(2\theta_x) \right)^{-1} \in \IC^{N_t \times
    N_t}, \\
    \widetilde{\mathcal{S}}_{\tau_L,h_L}^{x,\nu^x}(\theta_x) &:=
    \begin{pmatrix}
      \left(\hat{\mathcal{S}}_{\tau_L,h_L}(\omega_x,\theta_x)
      \right)^{\nu^x} & 0 \\
      0 & \left(\hat{\mathcal{S}}_{\tau_L,h_L}(\omega_x,\gamma(\theta_x))
      \right)^{\nu^x}
    \end{pmatrix} \in \IC^{2 N_t \times 2N_t},\\
    \widetilde{\mathcal{R}}_x(\theta_x) &:=
    \begin{pmatrix}
      \hat{\mathcal{R}}_x(\theta_x) I_{N_t} &
      \hat{\mathcal{R}}_x(\gamma(\theta_x)) I_{N_t}
    \end{pmatrix}  \in \IC^{2 N_t \times N_t},\\
        \widetilde{\mathcal{P}}_x(\theta_x) &:=
    \begin{pmatrix}
      \hat{\mathcal{P}}_x(\theta_x) I_{N_t} \\
      \hat{\mathcal{P}}_x(\gamma(\theta_x)) I_{N_t}
    \end{pmatrix}  \in \IC^{N_t \times 2N_t},
  \end{align*}
  and the Fourier symbols
  \begin{align*}
    \hat{A}_{\tau_L,h_L}(\theta_x) &:= \frac{h_L}{3}\left(2 +
    \cos(\theta_x)\right) K_{\tau_L} +
  \frac{2}{h_L}\left(1-\cos(\theta_x)\right) M_{\tau_L} \in
    \IC^{N_t \times N_t},\\
    \hat{\mathcal{S}}_{\tau_L,h_L}(\omega_x,\theta_x) &:= I_{N_t} - \omega_x
    \left( \frac{2h_L}{3} K_{\tau_L} + \frac{2}{h_L} M_{\tau_L} \right)^{-1}
    \hat{A}_{\tau_L,h_L}(\theta_x) \in \IC^{N_t \times N_t}.
  \end{align*}
  Hence we can analyze the modified two-grid iteration matrices
  (\ref{chap4_twoGridIterationMatrixApprox}) by taking the additional
  approximation with the mapping (\ref{chap4_mappingPropertyApprox})
  into account. For the smoothing steps $\nu_1^x = \nu_2^x = 2$ and
  the damping parameter $\omega_x = \frac{2}{3}$ for the spatial
  multigrid component, the theoretical
  convergence factors with semi coarsening in time are
  plotted in Figures \ref{chap4PlotMGSemiP0}--\ref{chap4PlotMGSemiP1}
  for the discretization parameter $\mu \in [10^{-6}, 10^6]$ with
  respect to the polynomial degrees $p_t \in \{0,1\}$. We observe that
  the theoretical convergence factors are always bounded by
  $\varrho(\overline{\mathcal{M}}_{\mu}^\mathrm{s}) \leq
  \frac{1}{2}$. We also notice that the theoretical convergence
  factors are a little bit larger for small discretization parameters
  $\mu$, compared to the case when the exact inverse of the diagonal
  matrix $D_{\tau_L,h_L}$ is used. The numerical factors are plotted
  as dots, triangles and squares in Figures
  \ref{chap4PlotMGSemiP0}--\ref{chap4PlotMGSemiP1}. We observe that
  the theoretical convergence factors coincide with the numerical
  results.

  In Figures \ref{chap4PlotMGFullP0}--\ref{chap4PlotMGFullP1} the
  convergence of the two-grid cycle for the full space-time coarsening
  case is studied. Here we see that the computed convergence factors
  are very close to the results which we obtained for the case when
  the exact inverse of the diagonal matrix $D_{\tau_L,h_L}$ is used.
\end{remark}

  \begin{figure}
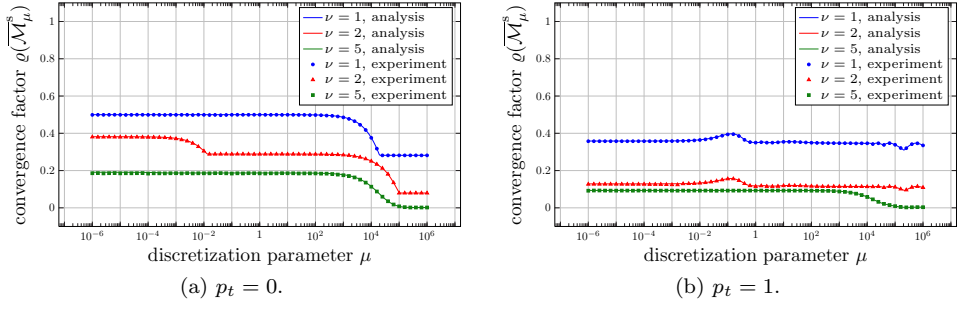

    \centering
    \subfloat[$p_t = 0$.]
    {
      \scalebox{0.43}{\input{./figures/plots/plotSTMGsemi0.tex}}
      \label{chap4PlotMGSemiP0}
    }\hfill
    \subfloat[$p_t = 1$.]
    {
      \scalebox{0.43}{\input{./figures/plots/plotSTMGsemi1.tex}}
      \label{chap4PlotMGSemiP1}
    }
    \caption{Average convergence factor
      $\varrho(\overline{\mathcal{M}}_{\mu}^\mathrm{s})$ for different
      discretization parameters $\mu$ and numerical convergence
      factors for $N_t = 256$ time steps and $N_x =
      1023$.}
  \end{figure}

  \begin{figure}
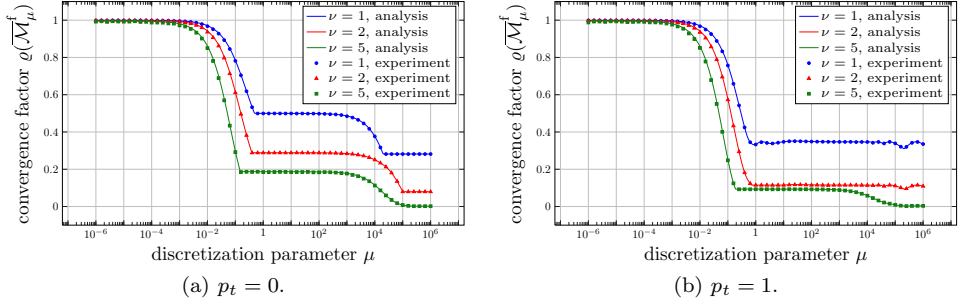

    \centering
    \subfloat[$p_t = 0$.]
    {
      \scalebox{0.43}{\input{./figures/plots/plotSTMGfull0.tex}}
      \label{chap4PlotMGFullP0}
    }\hfill
    \subfloat[$p_t = 1$.]
    {
      \scalebox{0.43}{\input{./figures/plots/plotSTMGfull1.tex}}
      \label{chap4PlotMGFullP1}
    }
    \caption{Average convergence factor
      $\varrho(\overline{\mathcal{M}}_{\mu}^\mathrm{f})$ for different
      discretization parameters $\mu$ and numerical convergence
      factors for $N_t = 256$ time steps and $N_x =
      1023$.}
  \end{figure}

\section{Numerical examples}\label{NumExSec}

We present now numerical results for the multigrid version of our
algorithm, for which we analyzed the two-grid cycle in Section
\ref{chap4_twoGridAnalysis}. Following our two-grid analysis, we also
apply full space-time coarsening only if $\mu_L \geq \mu^\ast$ in the
multigrid version. If $\mu_L < \mu^\ast$, we only apply semi
coarsening in time. In that case, we will have for the next coarser
level $\mu_{L-1} = 2 \tau_L h_L^{-2} = 2 \mu_L$.
This implies that the discretization parameter $\mu_{L-1}$ gets larger when semi
coarsening in time is used. Hence, if
$\mu_{L-k} \geq \mu^\ast$ for $k < L$ we can apply full space-time coarsening
to reduce the computational costs. If full space-time coarsening is applied, we
have $\mu_{L-1} = 2 \tau_L \left(2 h_L \right)^{-2} = \frac{1}{2} \mu_L$,
which results in a smaller discretization parameter $\mu_{L-1}$. We therefore
will combine semi coarsening in time or full space-time coarsening in the right
way to get to the next coarser space-time level. For
different discretization parameters $\mu = c \mu^\ast$, $c \in
\{1,10\}$, this coarsening strategy is shown in Figure
\ref{chap4figureRestriction} for 8 time and 4 space levels. The
restriction and prolongation operators for the space-time multigrid scheme are
then defined by the given coarsening strategy.
\begin{figure}
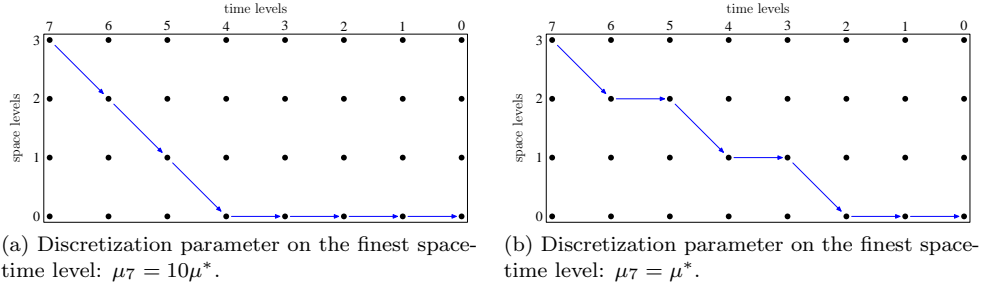

    \centering
    \subfloat[Discretization parameter on the finest space-time level: $\mu_7 =
    10 \mu^\ast$.]
    {
      \scalebox{0.44}{\includegraphics{./figures/spaceTimeRestriction.1}}
    }\hfill
    \subfloat[Discretization parameter on the finest space-time level: $\mu_7 =
    \mu^\ast$.]
    {
      \scalebox{0.44}{\includegraphics{./figures/spaceTimeRestriction.2}}
    }
    \caption{Space-time coarsening for different discretization parameters
      $\mu_L$.}\label{chap4figureRestriction}
\end{figure}

We now show examples to illustrate the performance of our new space-time
multigrid method.
\begin{example}[Multigrid iterations]\label{chap4_example1}
  In this example we consider the spatial domain $\Omega = (0,1)^3$
  and the simulation interval $(0,T)$ with $T = 1$. The initial
  decomposition for the spatial domain $\Omega$ is given by $12$
  tetrahedra. We use several uniform refinement levels to study
  the convergence behavior of the space-time multigrid solver with
  respect to the space-discretization. For the coarsest time level we
  use one time step, i.e. $\tau_0 = 1$. For the space discretization we use P1 conforming finite elements and for the time discretization we use piecewise linear
  discontinuous ansatz functions, i.e. $p_t = 1$. To test the performance
of the space-time multigrid method we use a zero right hand side, i.e.
$\vec{f} = \mathbf{0}$ and as an
initial guess $\vec{u}^0$ we use a random vector with values between
zero and one. For the space-time multigrid solver we use $\nu_1 =
\nu_2 = 2$, $\omega_t = \frac{1}{2}$, $\gamma = 1$.  We apply for each
block $A_{\tau_L,h_L}$ one geometric multigrid V-cycle to approximate
the inverse of the diagonal matrix $D_{\tau_L,h_L}$. For this
multigrid cycle we use $\nu_1^x = \nu_2^x = 2$, $\omega_x =
\frac{2}{3}$, $\gamma_x = 1$.  For the smoother we use a damped block
Jacobi smoother. We apply the space-time multigrid solver until we
have reached a given relative error reduction of
$\varepsilon_{\mathrm{MG}} = 10^{-8}$. In Table
\ref{chap4_TableExample1}, the iteration numbers for several space and
time levels are given. We observe that the iteration numbers stay
bounded independently of the mesh size $h_{L_x}$, the time step size
$\tau_{L_t}$ and the number of time steps $N_{L_t} = 2^{L_t}$.
\begin{table}
  \begin{center}
      \begin{tabular}{cc|ccccccccccccccc} \thickhline
        & & & & \multicolumn{11}{c}{time levels} \\
        & & $0$ & $1$ & $2$ & $3$ & $4$ & $5$ & $6$ & $7$ & $8$ & $9$ & $10$ &
        $11$ & $12$ & $13$ & $14$ \\ \hline
        \multirow{6}{1em}{\begin{sideways} space levels\end{sideways}} & $0$ &
        $1$ & $7$ & $7$ & $7$ & $7$ & $7$ & $7$ & $7$ & $8$ & $8$ & $9$ & $9$ & $9$ & $9$ & $9$ \\
        & $1$ & $1$ & $7$ & $7$ & $7$ & $7$ & $7$ & $7$ & $7$ & $8$ & $8$ & $9$ & $9$ & $9$ & $9$ & $9$\\
        & $2$ & $1$ & $7$ & $7$ & $7$ & $7$ & $7$ & $8$ & $7$ & $8$ & $8$ & $9$ & $9$ & $9$ & $9$ & $9$\\
        & $3$ & $1$ & $7$ & $7$ & $7$ & $7$ & $8$ & $8$ & $8$ & $8$ & $8$ & $9$ & $9$ & $9$ & $9$ & $9$\\
        & $4$ & $1$ & $7$ & $7$ & $7$ & $8$ & $8$ & $8$ & $8$ & $8$ & $8$ & $8$ & $9$ & $9$ & $9$ & $9$\\
        & $5$ & $1$ & $7$ & $7$ & $7$ & $7$ & $7$ & $8$ & $8$ & $8$ & $8$ & $8$
        & $9$ & $9$ & $9$ & $9$ \\ \thickhline
      \end{tabular}
  \end{center}
  \caption{ Multigrid iterations for Example
    \ref{chap4_example1}.}\label{chap4_TableExample1}
\end{table}
\end{example}

\begin{example}[High order time discretizations]\label{chap4_example2}
  In this example we study the convergence of the space-time
  multigrid method for different polynomial degrees $p_t$,
  which are used for the underlying time discretization. To do so, we
  consider the spatial domain $\Omega = (0,1)^2$ and the
  simulation interval $(0,T)$ with $T = 1024$. For the space-time
  discretization we use tensor product space-time elements with
  piecewise linear continuous ansatz functions in space, and for the
  discretization in time we use a fixed time step size $\tau =
  1$. For the initial triangulation of the spatial domain $\Omega$ we
  consider $4$ triangles, which are refined uniformly several
  times. For the space-time multigrid approach we use the same
  parameters as in Example \ref{chap4_example1}. We solve the linear
  system (\ref{chap4_linearSystem}) with zero right hand side,
  i.e. $\vec{f} = \mathbf{0}$ and for the initial vector $\vec{u}^0$
  we use a random vector with values between zero and one. We
  apply the space-time multigrid solver until we have reached a
  relative error reduction of $\varepsilon_{\mathrm{MG}} =
  10^{-8}$. In Table \ref{chap4_TableExample2} the iteration numbers
  for different polynomial degrees $p_t$ and different space levels
  are given. We observe that the iteration numbers are bounded,
  independently of the ansatz functions for the time
  discretization.
  \begin{table}
    \centering
      \begin{tabular}{cc|cccccccccccccc}
        \thickhline
        & & \multicolumn{14}{c}{polynomial degree $p_t$} \\
        & & $0$ & $1$ & $2$ & $3$ & $4$ & $5$ & $10$ & $15$ & $20$ & $25$ &
        $30$ & $35$ & $40$ & $45$ \\ \hline
        \multirow{8}{1em}{\begin{sideways} space levels\end{sideways}} & $0$ &
        $7$ & $7$ & $6$ & $6$ & $6$ & $6$ & $5$ & $5$ & $4$ & $4$ & $4$ & $4$ &
        $5$ & $5$ \\
        & $1$ & $7$ & $7$ & $7$ & $7$ & $7$ & $7$ & $7$ & $7$ & $7$ & $7$ & $7$ & $7$ & $7$ & $7$ \\
        & $2$ & $7$ & $7$ & $7$ & $7$ & $7$ & $7$ & $7$ & $7$ & $7$ & $7$ & $7$ & $7$ & $7$ & $7$ \\
        & $3$ & $7$ & $7$ & $7$ & $7$ & $7$ & $7$ & $7$ & $7$ & $7$ & $7$ & $7$ & $7$ & $7$ & $7$ \\
        & $4$ & $7$ & $7$ & $7$ & $7$ & $7$ & $7$ & $7$ & $7$ & $7$ & $7$ & $7$ & $7$ & $7$ & $7$ \\
        & $5$ & $7$ & $7$ & $7$ & $7$ & $7$ & $7$ & $7$ & $7$ & $7$ & $7$ & $7$ & $7$ & $7$ & $7$ \\
        & $6$ & $7$ & $7$ & $7$ & $7$ & $7$ & $7$ & $7$ & $7$ & $7$ & $7$ & $7$ & $7$ & $7$ & $7$ \\
        & $7$ & $7$ & $7$ & $7$ & $7$ & $7$ & $7$ & $7$ & $7$ & $7$ & $7$ & $7$
        & $7$ & $7$ & $7$ \\ \thickhline
      \end{tabular}
    \caption{Multigrid iterations with respect to the polynomial degree $p_t$.}\label{chap4_TableExample2}
  \end{table}

\end{example}

\section{Parallelization}\label{ParallelSec}

One big advantage of our new space-time multigrid method is that it
can be parallelized also in the time direction, i.e.  the damped block
Jacobi smoother can be executed in parallel in time. For each
time step we have to apply one multigrid cycle in space to
approximate the inverse of the diagonal matrix $D_{\tau_L, h_L}$. The
application of this space multigrid cycle can also be done in
parallel, where one may use parallel packages like in
\cite{Falgout2002, Falgout2006, Heppner2013}. Hence the problem
(\ref{chap4_linearSystem}) can be fully parallelized in 
space and time, see Figure \ref{chap4figureParallelization}.
\begin{figure}
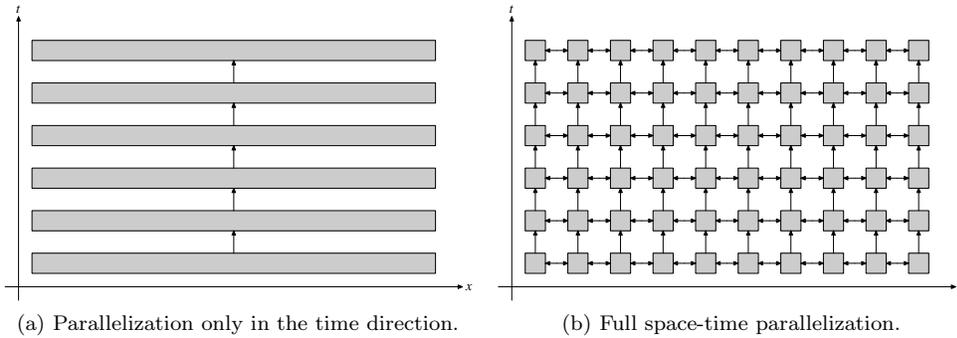

    \centering
    \subfloat[Parallelization only in the time direction.]
    {
      \scalebox{0.44}{\includegraphics{./figures/Space_Time_Parallelisazion.1}}
    }\hfill
    \subfloat[Full space-time parallelization.]
    {
      \scalebox{0.44}{\includegraphics{./figures/Space_Time_Parallelisazion.2}}
    }
    \caption{Communication pattern on a fixed
      level.}\label{chap4figureParallelization}
\end{figure}

The next example shows the excellent weak and strong scaling
properties of our new space-time multigrid method.
\begin{example}[Parallel computations]
In this example we consider the spatial domain $\Omega = (0,1)^3$,
which is decomposed into $49\;152$ tetrahedra. For the discretization
in space we use P1 conforming finite elements and for
the time discretization we use polynomials of order $p_t = 3$ and a fixed
time step size $\tau = 10^{-1}$. For the multigrid solver in space we use the best possible settings such that we obtain the smallest computational times when we apply the usual forward substitution. In particular we use a damped Gau\ss-Seidel smoother with the damping parameter $\omega_x = 1.285$ and we apply one pre- and one post-smoothing step, i.e. $\nu_1^x = \nu_2^x = 1$. We also tune the multigrid parameters with respect to time, such that we also obtain the best possible computational times for the presented space-time multigrid solver. Here we use $\nu_1 = \nu_2 = 1$ smoothing steps and since $\mu = \tau h^{-2}$ is large enough we use for the damping parameter $\omega_t = 1$, see also Remark \ref{remarkDampingParameter}. Of course we could also use the assymptotic optimal damping parameter $\omega_t^\ast = \frac{1}{2}$ which would lead to slightly more multigrid iterations for this case.

To show the parallel performance, we first study the weak scaling
behavior of the new multigrid method. We use a fixed number of time
steps per core, i.e. $2$ time steps for each core, and we increase the
number of cores when we increase the number of time steps. Hence the
computational cost for one space-time multigrid cycle stays almost the
same for each core. Only the cost for the communication grows, since
the space-time hierarchy gets bigger, when we increase the number of
time steps. In Table \ref{chap4_TableWeak}, we give timings for
solving the linear system (\ref{chap4_linearSystem}) for a different
number of time steps.  We see that the multigrid iterations stay
bounded, if we increase the problem size and that the computational
costs stay completely constant if we increase the number of cores. We
also compare the presented space-time multigrid solver with the usual
forward substitution. For this we apply for each time step the space
multigrid solver with the best possible settings from above. We run
the space multigrid solver until we obtain the same relative error
tolerance as for the space-time multigrid method. In Table
\ref{chap4_TableWeak} the timings for the forward substitution are
compared with the parallel space-time multigrid solver. Here we
observe that the space-time multigrid solver is already faster when we
use only two cores. Furthermore, when we increase the number of cores
we observe that the space-time multigrid approach completely
outperforms the forward substitution.

To test the strong scaling behavior, we fix the problem size and use
$32\;768$ time steps, which results in a linear system with
$979\;238\;912$ unknowns. Then we increase the number of cores, which
results in smaller and smaller problems per computing core. In Table
\ref{chap4_TableStrong} the timings are given for different
numbers of cores. We see that the computational costs are divided
by a factor very close to two, if we double the number of cores.
All the parallel computations of this example were performed on
  the Lemanicus BlueGene/Q Supercomputer in Lausanne, Switzerland,
  and for one, two and four cores, the computational times needed
  were too large to run, since Lemanicus has a maximum wall
    clock time restriction of 24 hours.
\begin{figure}
    \centering
    \subfloat[Weak scaling results.]
    {
      \scalebox{0.65}{
      \begin{tabular}{r|r|r|c|r|r}
        \thickhline
        cores & time steps & \multicolumn{1}{c|}{dof} & iter & time & fwd. sub. \\ \hline
        $1$ & $2$ & $59\;768$ & $7$ & $29.3$ & $21.2$\\
        $2$ & $4$ & $119\;536$ & $7$ & $29.6$ & $44.0$\\
        $4$ & $8$ & $239\;072$ & $7$ & $29.6$ & $87.7$\\
        $8$ & $16$ & $478\;144$ & $7$ & $29.7$ & $176.7$ \\
        $16$ & $32$ & $956\;288$ & $7$ & $29.6$ & $351.7$ \\
        $32$ & $64$ & $1\;912\;576$ & $7$ & $29.7$ & $703.8$\\
        $64$ & $128$ & $3\;825\;152$ & $7$ & $29.7$ & $1\;408.3$\\
        $128$ & $256$ & $7\;650\;304$ & $7$ & $29.7$ & $2\;819.8$ \\
        $256$ & $512$ & $15\;300\;608$ & $7$ & $29.8$ & $5\;662.7$\\
        $512$ & $1\;024$ & $30\;601\;216$ & $7$ & $29.7$ & $11\;278.4$ \\
        $1\;024$ & $2\;048$ & $61\;202\;432$ & $7$ & $29.8$ & $22\;560.3$ \\
        $2\;048$ & $4\;096$ & $122\;404\;864$ & $7$ & $29.7$ & $45\;111.3$ \\
        $4\;096$ & $8\;192$ & $244\;809\;728$ & $7$ & $29.7$ & $87\;239.8$ \\
        $8\;192$ & $16\;384$ & $489\;619\;456$ & $7$ & $29.8$ & $174\;283.5$ \\
        $16\;384$ & $32\;768$ & $979\;238\;912$ & $7$ & $29.7$ & $348\;324.0$\\ \thickhline
      \end{tabular}
      }\label{chap4_TableWeak}
    }\hfill
    \subfloat[Strong scaling results.]
    {
      \scalebox{0.65}{\begin{tabular}{r|r|r|c|r}
        \thickhline
        cores & time steps & \multicolumn{1}{c|}{dof} & iter & time \\ \hline
        $1$ & $32\;768$ & $979\;238\;912$ & $7$ & $-$ \\
        $2$ & $32\;768$ & $979\;238\;912$ & $7$ & $-$ \\
        $4$ & $32\;768$ & $979\;238\;912$ & $7$ & $-$ \\
        $8$ & $32\;768$ & $979\;238\;912$ & $7$ & $60\;232.5$ \\
        $16$ & $32\;768$ & $979\;238\;912$ & $7$ & $30\;239.5$ \\
        $32$ & $32\;768$ & $979\;238\;912$ & $7$ & $15\;136.2$ \\
        $64$ & $32\;768$ & $979\;238\;912$ & $7$ & $7\;590.0$ \\
        $128$ & $32\;768$ & $979\;238\;912$ & $7$ & $3\;805.2$ \\
        $256$ & $32\;768$ & $979\;238\;912$ & $7$ & $1\;906.8$ \\
        $512$ & $32\;768$ & $979\;238\;912$ & $7$ & $951.8$ \\
        $1\;024$ & $32\;768$ & $979\;238\;912$ & $7$ & $473.4$ \\
        $2\;048$ & $32\;768$ & $979\;238\;912$ & $7$ & $238.2$ \\
        $4\;096$ & $32\;768$ & $979\;238\;912$ & $7$ & $119.0$ \\
        $8\;192$ & $32\;768$ & $979\;238\;912$ & $7$ & $59.4$ \\
        $16\;384$ & $32\;768$ & $979\;238\;912$ & $7$ & $29.7$ \\ \thickhline
      \end{tabular}
      }\label{chap4_TableStrong}
    }
    \caption{Scaling results with solving times in seconds.}
  \end{figure}

  \end{example}

\section{Conclusions}\label{ConclusionSec}

We presented a new space-time multigrid method for the heat equation,
and used local Fourier mode analysis to give precise asymptotic
convergence and parameter estimates for the two-grid cycle. We showed
that this asymptotic analysis predicts very well the performance of
the new algorithm, and our parallel implementation gave excellent
weak and strong scaling results for a large number of processors.

This new space-time multigrid algorithm can not only be used for the
heat equation, it is applicable to general parabolic problems. It has
successfully been applied to the time dependent Stokes equations, where
one obtains similar speed up results as for the heat equation. Furthermore,
this technique has been applied successfully to parabolic control problems,
but the analysis and the results will appear elsewhere. 

\section*{Acknowledgments}

The financial support for CADMOS and the Blue Gene/Q system is
provided by the Canton of Geneva, Canton of Vaud, Hans Wilsdorf
Foundation, Louis-Jeantet Foundation, University of Geneva, University
of Lausanne, and Ecole Polytechnique F\'ed\'erale de Lausanne.


\bibliographystyle{abbrv}
\bibliography{bibliography}

\end{document}